\documentclass[preprint,aos]{imsart}
\usepackage{amsmath,amssymb,amsthm}
\usepackage[mathscr]{eucal}
\usepackage{mathtools}
\usepackage[math]{easyeqn}
\usepackage{srcltx}
\usepackage[numbers]{natbib}
\usepackage{comment}
\usepackage{etoolbox}
\usepackage{mathrsfs}
\usepackage[usenames, dvipsnames]{color}
\usepackage{xfrac}
\usepackage{todonotes}
\usepackage{multirow}
\usepackage[autostyle]{csquotes}
\usepackage{booktabs}
\usepackage{longtable}
\usepackage{xr}
\usepackage{bbm}
\usepackage{bm}
\usepackage{enumitem}

\renewcommand\cite{\citet}
\RequirePackage[colorlinks,citecolor=blue,urlcolor=blue]{hyperref}

\arxiv{arXiv:1611.02686}

\startlocaldefs
\newtheorem{lemma}{Lemma}[section]
\newtheorem{theorem}{Theorem}[section]

\newtheorem{corollaryT}{Corollary}[section]
\theoremstyle{remark}
\newtheorem{remark}{Remark}[section]
\theoremstyle{newremark}

\numberwithin{newremark}{section}
\numberwithin{equation}{section}
\endlocaldefs

\newcommand*{\Scale}[2][4]{\scalebox{#1}{$#2$}}

\def\KL{\operatorname{KL}}

\newcommand{\supp}{\operatorname{supp}}

\def \Ps{\P^{*}}

\def \zv{z}

\def \Zy{Z}
\def \Uy{U}

\def \Sigmaz {\Sigma_{\zv}}
\def \Sigmazi{\Sigma_{\zv,i}}

\newcommand{\sbt}{\,\begin{picture}(0,1)
                    \put(1,3){\circle{2}}
                    \put(1,3){\circle{3}}
                    \end{picture}\ }

\newcommand{\sbtt}{\,\begin{picture}(0,1)
                    \put(1,2){\circle{2}}
                    \put(1,2){\circle{3}}
                    \end{picture}\ }

\def\Lb{L^{\sbt}}
\def\Eb{\E^{\sbt}}
\def\Pb{\P^{\sbt}}

\def \BCl{\mathscr{B}}
\def \HCl{\mathscr{H}}

\def \Deltab{\Delta^{\sbt}}

{

\def\Xs{\mathcal{X}_{0}}

\def \Id {\boldsymbol{I}}

\def\thetavt{\tilde{\thetav}}

\def\thetavbt{\tilde{\thetav}^{\sbtt}}

\newcommand{\cc}[1]{\mathscr{#1}}
\newcommand{\bb}[1]{\boldsymbol{#1}}
\renewcommand{\(}{$\,}
\renewcommand{\)}{\,$}

\def\argmax{\operatornamewithlimits{argmax}}
\def\argmin{\operatornamewithlimits{argmin}}
\def \xx {x}
\def\thetas{\theta^{\ast}}
\def \thetavs{\thetas}
\def \thetav{\theta}
\def \thetas{\theta_{0}}

\def \veps{\varepsilon}
\def \vepst{\varepsilon}

\def \epsc{\varepsilon}

\def \ex{e}

\def\Var{\operatorname{Var}}
\def\T{\top}
\def\LL{\cc{L}}
\def\cond{\, \big| \,}
\def\R{\mathbb{R}}
\def\E{\mathbf{E}\hspace{0.01cm}}
\def\P{\mathbf{P}}
\def \Id{\mathbf{I}}
\def \Ind{\mathbbm{1}}
\def \dimp{p}

\def\CS{\cc{E}}
\def\gmu{\mathfrak{a}}

\def \ys{\bb{y}}
\def \SPDp {S_{\dimp}^{\Scale[0.65]{+}}}

\DeclareMathOperator{\vect}{vec}
\def \Sb {S^{\sbt}}

\def \Xb {X^{\sbt}}

\def \QuantL {Q_{L}}
\def \Quantb{Q^{\sbt}}
\def \Quants{Q^{\ast}}
\def \QuantLb {Q^{\sbt}_{L}}

\def \St{\tilde{S}}
\def \cKS {C_{\Sigma,K}}
\def \iv{\bm{i}}

\def \Cz{C_{\zv}}
\def \Cx{C_{X}}

\def \Czxiv{C_{\zv,L}}
\def \CzL{\bar{C}_{\zv}}

\def\SigmazL{\bar{\Sigma}_{\zv}}

\def \Xs{X^{\ast}}

\def \Ps{\Pb}
\def \Es{\Eb}
\def \Ss {S^{\ast}}
\def \Xbar{\bar{X}}
\def \Xbars{\bar{X}^{\ast}}

\def \Deltab{\Delta^{\sbt}}
\def \Deltas{\Delta^{\ast}}

\def \Xmean{\bar{X}}
\def \Xs{X^{\ast}}
\def \Pb{\P^{\ast}}
\def \Eb{\E^{\ast}}
\def \Ss{S^{\ast}}

\def\eqdef{\coloneqq}

\def \bmC{C}
\def \CBEBiid{\bmC_{\BCl,iid}}
\def \CBEBniid{\bmC_{\BCl,ind}}
\def \CBEBniidw{\bmC_{\BCl,w,ind}}

\definecolor{brown(traditional)}{rgb}{0.59, 0.29, 0.0}
\definecolor{hookersgreen}{rgb}{0.0, 0.44, 0.0}
\definecolor{magenta(process)}{rgb}{1.0, 0.0, 0.56}
\definecolor{folly}{rgb}{1.0, 0.0, 0.31}
\definecolor{brightgreen}{rgb}{0.4, 1.0, 0.0}

\makeatletter
\newcommand{\labitem}[2]{%
\def\@itemlabel{\textbf{#1}}
\item
\def\@currentlabel{#1}\label{#2}}

\makeatother

\begin{document}

\begin{frontmatter}

\title{Nonclassical Berry--Esseen inequalities and accuracy of the bootstrap}

\runtitle{Berry--Esseen inequalities and accuracy of bootstrap}

 \author{\fnms{Mayya} \snm{Zhilova}\thanksref{a1}
 \corref{}\ead[label=e1]{mzhilova@math.gatech.edu}}
\thankstext{a1}{Support by the National Science Foundation grant DMS-1712990 is gratefully acknowledged} 
\address{School of Mathematics\\Georgia Institute of Technology\\
Atlanta, GA 30332-0160 USA \\\printead{e1}}

\runauthor{M. Zhilova}

\affiliation{School of Mathematics\\Georgia Institute of Technology}
\begin{abstract}
We study accuracy of bootstrap procedures for estimation of quantiles of a smooth function of a sum of independent sub-Gaussian random vectors. We establish higher-order approximation bounds with error terms depending on a sample size and a dimension explicitly. These results  lead to improvements of  accuracy of a weighted bootstrap procedure for general log-likelihood ratio statistics. The key element of our proofs of the bootstrap accuracy is a multivariate higher-order Berry--Esseen inequality. We consider a problem of approximation of  distributions of two sums of zero mean independent random vectors, such that summands with the same indices have equal moments up to at least the second order. The derived approximation bound is uniform on the sets of all Euclidean balls. The presented approach extends classical Berry--Esseen type inequalities to higher-order approximation bounds. The theoretical results are illustrated with numerical experiments.
\end{abstract}


\end{frontmatter}

\section{Introduction}
In this paper we study accuracy of bootstrap procedures for estimation of quantiles of statistics of the form \(\|S_n\|\) or $f(S_{n})$, where $\| \cdot\|$ denotes the $\ell_2$-norm, and $f(\cdot): \R^{\dimp}\mapsto \R$ is a twice continuously differentiable function with bounded second derivative,  
\begin{equation*}
S_n\eqdef n^{-1/2}{{\textstyle\sum\nolimits_{i=1}^{n}}}X_i
\end{equation*}
for independent sub-Gaussian random vectors $X_1,\dots, X_n \in \R^{\dimp}$ with positive definite covariance matrices $\Var (X_{i})$ $\forall\,i=1,\dots,n$. 
We consider the non-asymptotic setting, when the leading approximation errors depend on $n$ and the dimension $\dimp$ explicitly. This setting allows to assess accuracy and limitations of a bootstrap approximation in terms of the dimension $\dimp$ and the sample size $n$. 
 Estimation of distribution of statistics of the types \(\|S_n\|\) or $f(S_{n})$ is necessary for construction of confidence sets and hypothesis testing in some important statistical models and problems, such as linear regression model with unknown distribution of errors,  general log-likelihood ratio statistic, construction of confidence sets for multivariate sample mean. 

We focus on two basic bootstrapping procedures. The first method considered here is the Efron's bootstrap (introduced by \cite{Efron1979bootstrap} in 1979), where the resampling is performed uniformly at random with replacement from the i.i.d. data $X_{1},\dots,X_{n} \in \R^{\dimp}$. In this case the bootstrap samples $\Xs_{1},\dots,\Xs_{n}$ have the distribution $\Pb(\Xs_{j}=X_{i}-\Xmean)=1/n\ \forall i,j=1,\dots,n$, where $\Xmean=n^{-1}\sum_{i=1}^{n}X_{i}$, and \(\Pb(\cdot)\eqdef\P(\cdot\cond X_{1},\dots,X_{n})\). Define for the  sum \(S_{n}\) its bootstrap version:
\begin{equation*}
\Ss_{n}\eqdef  n^{-1/2}{{\textstyle\sum\nolimits_{i=1}^{n}}}\Xs_{i}.
\end{equation*}
One of the main results of the paper is the following uniform approximation bound on the set $\BCl$ of all Euclidean balls in $\R^{\dimp}$ which holds with high probability: 
\begin{align}
\label{intro:bootstappr1}
\sup\nolimits_{B\in \BCl}\left|
 \P\left(S_{n}\in B \right)
-
\Pb\bigl(\Ss_{n}\in B \bigl)
\right|&\leq C_{\sigma,K,z} \bigl\{\dimp^{{K}/{(K-2)}}/{n}\bigr\}^{1/2},
\end{align}
where $K\geq 3$ is a natural number, and constant $C_{\sigma,K,z}$ depends (up to $\log$-terms) on $K$, on value $\Cz$ introduced in  \eqref{def:Cz_mainres} in Section \ref{SECT:MAINRES_BE}, and on constant $\sigma^{2}>0$ which comes from the following condition on the moment generating function of $X_{i}$: 
$\E\bigl\{\exp(\alpha^{\T}X_{i})\bigr\}\leq
\exp \left(\|\alpha\|^{2}\sigma^{2}/2\right)\  \forall \alpha\in \R^{\dimp}$
(see also Remark \ref{remark:thb1} in Section \ref{sect:someremarks} for an asymptotic version of the statement).

The second  of the considered methods is the weighted bootstrap. Here $X_{1}, \dots, X_{n} \in \R^{\dimp}$ are assumed to be zero mean, independent but not necessarily identically distributed. Introduce the following random variables
\begin{equation}
\label{def:wb_eps_intro}
\begin{gathered}
\text{\(\veps_{1},\dots,\veps_{n}\in \R\), i.i.d., independent of \(\{X_{i}\}_{i=1}^{n}\)},\\
\E\veps_{i}=0,\ 
\E(\veps_{i}^{2})=1,\ \E(\veps_{i}^{3})=1,\ 
\E(\veps_{i}^{4})<\infty.
\end{gathered}
\end{equation}
The weighted or the multiplier bootstrap approximation of \(S_{n}\) is:
\begin{equation}
\label{def:SbnW}
\Sb_{n}\eqdef  n^{-1/2}{{\textstyle\sum\nolimits_{i=1}^{n}}}X_{i} \veps_{i}.
\end{equation}
For this version of the bootstrap estimator we derive the following bound which holds with high probability: 
\begin{align}
\label{intro:bootstapprw}
\sup\nolimits_{B\in \BCl}\left|
 \P\left(S_{n}\in B \right)
-
\Pb\bigl(\Sb_{n}\in B \bigl)
\right|& \leq C_{\sigma,z}(\dimp^{2}/{n})^{1/6}
\end{align}
for  $\dimp \leq C \sqrt{n}$, where constant $C_{\sigma,z}$ depends on value $\CzL$ introduced in \eqref{def:Sigmazi} in Section \ref{SECT:MAINRES_BE}, 
 and on constants $\sigma_{i}^{2}>0$ which come from conditions on m.g.f.-s of $X_{i}$  for each $i=1,\dots, n$ (an asymptotic version of the statement is given in Remark \ref{remark:thb1}, Section \ref{sect:someremarks}). 
Bounds \eqref{intro:bootstappr1} and \eqref{intro:bootstapprw} imply, in particular, that if the random vector $X_{1}$ is sub-Gaussian, and the ratio $\dimp^{{K}/{(K-2)}}/{n}$ (or $\dimp^{2}/n$ for bound \eqref{intro:bootstapprw}) is rather small, then the bootstrap approximation is accurate.  
In addition, we give an example of $\{X_{i}\}_{i=1}^{n}$ which justifies that the condition $\dimp=o(n)$ (or $\dimp=o(n^{1/2})$ for the weighted bootstrap) as $n\to \infty$ is necessary for the consistency result $\sup_{B\in \BCl}\bigl|
 \P\left(S_{n}\in B \right)
-
\Pb\bigl(\Ss_{n}\in B \bigl)
\bigl|\overset{\P}{\to} 0$ (or $\sup_{B\in \BCl}\bigl|
 \P\left(S_{n}\in B \right)
-
\Pb\bigl(\Sb_{n}\in B \bigl)
\bigr|\overset{\P}{\to} 0$
for the weighted bootstrap method) as $n\to \infty$.

An important feature of the present results is that they do not involve any asymptotic  methods such as, for example, Edgeworth expansions that are frequently employed for studying the rates of convergence of bootstrap estimators. We develop a new non-asymptotic approach that allows to study higher-order accuracy of bootstrap in high-dimensional setting. The key element in the proofs of our theoretical results about bootstrapping is a multivariate Berry--Esseen inequality in a nonclassical form which might be interesting by itself. 

  We consider the problem of approximation of a probability distribution of the sum $S_n=n^{-1/2}\sum\nolimits_{i=1}^{n}X_i$, where  $X_i\in\R^{\dimp}$ are independent random vectors such that $\E X_i =0$ and $\E (\|X_i\|^K)<\infty$ for some $K\geq 3$. The approximating distribution corresponds to the sum 
$\St_{n}\eqdef n^{-1/2}\sum\nolimits_{i=1}^{n} Y_{i}$, 
where $Y_1,\dots, Y_n\in\R^{\dimp}$ are independent random vectors,  independent of \(\{X_{i}\}_{i=1}^{n}\)  such that  $\E (\|Y_i\|^K)<\infty$, 
\begin{gather}
\label{cond:XiYi}
\E (X_{i}^{k})=\E (Y_{i}^{k})\ \forall k=1,\dots,K-1,
\end{gather}
and  \(Y_i {=}\Zy_i+\Uy_i\) for some independent random vectors \(\Zy_{i},\Uy_{i}\in\R^{\dimp}\), where \( \Zy_{i}\) are normally distributed with \(\E\Zy_{i}=0\). Throughout the paper  the condition $\E \bigl(X^{k}\bigr)=\E \bigl(Y^{k}\bigr)\  \forall\,k=1,\dots,K$ on the higher-order moments of random vectors $X=(x_{1},\dots,x_{\dimp})^{\T}\in\R^{\dimp}$ and  $Y=(y_{1},\dots,y_{\dimp})^{\T}\in\R^{\dimp}$ denotes that for all degrees $k=1,\dots,K$ and for all indices $1\leq i_{1},\dots, i_{k}\leq \dimp$ 
\begin{gather}
\label{eq:mom}
\E(x_{i_1} \dots  x_{i_k})=\E(y_{i_1} \dots  y_{i_k}).
\end{gather}
In Lemma \ref{lemma:Yexist} we show that if a cardinality of a support of  \(X_{i}\) is sufficiently large, then  
the corresponding random vectors \(\Zy_{i},\Uy_{i}\) always exist. The probability distribution of such constructed random vector $\St_{n}$  turns out to be a rather good approximation of a distribution of the initial sum $S_{n}$. One of the main results in the paper is the following uniform Berry--Esseen type bound:
 for the set $\BCl$ of all Euclidean balls in $\R^{\dimp}$ and for i.i.d. $\{X_{i}\}_{i=1}^{n}$
\begin{align}
\nonumber
&
\hspace{-2cm}
\sup\nolimits_{B \in \BCl}
\bigl|
\P(S_{n}\in B)-\P(\St_{n}\in B)
\bigr|
\\
&\leq
 \cKS
 { \left\{\E\left(\|X_{1}\|^{K}+\|Y_{1}\|^{K}\right)\right\}^{{1}/{(K-2)}}}{n^{-1/2}},
 \label{ineq:BEintro}
 \end{align}
where constant $\cKS$ depends on $K$ and on eigenvalues of $(\Var Z_{1})^{-1}.$  
Bound  \eqref{ineq:BEintro} includes the classical Berry--Esseen inequality, where the approximating distribution is  multivariate normal i.e. $Y_{i}\sim \mathcal{N}(0,\Var X_{i})$ and $K=3$. If $K>3$, this bound  exploits more information about coinciding moments, than the normal approximation does, which leads to a better accuracy. 

Our proof of bound \eqref{ineq:BEintro} is based on the work of \cite{Bentkus2003BE}, where the author obtained a multivariate Berry--Esseen inequality involving the standard normal distribution, uniformly on the set of all Euclidean balls, and also on the set of all convex sets in $\R^{\dimp}$. 
 In this paper we extend the proof in the work of \cite{Bentkus2003BE} to the \enquote{quasi-normal} case, i.e. for the approximation with the sum $\St_{n}$ of the convolutions \(Y_{i}=\Zy_{i}+\Uy_{i}\), where \(Z_{i}\) are normally distributed. This approach allows us to use both the properties of the normal distribution and the higher moments condition \eqref{cond:XiYi}. 
 Furthermore, if $\|X_{1}\|^{2}\leq \dimp$ a.s., then inequality \eqref{ineq:BEintro} implies
 $$
 \sup\nolimits_{B \in \BCl}
\bigl|
\P(S_{n}\in B)-\P(\St_{n}\in B)
\bigr|\leq  \cKS \bigl\{\dimp^{{K}/{(K-2)}}/{n}\bigr\}^{1/2}.
 $$
In Lemma \ref{lemma:BElowerbound} in Section \ref{SECT:MAINRES_BE}, we show that for $K\geq 3$ the requirement $\dimp=o(n^{(K-2)/K})$ as $n\to \infty$ is necessary for $\sup_{x\in\R}|\P(\|S_{n}\|\leq x) -\P(\|\tilde{S}_{n}\|\leq x) | \to 0$, $n\to \infty$ for some approximating distribution $\tilde{S}_{n}$, satisfying conditions of Theorem \ref{theor:NcBE}.

Now let us discuss how Berry--Esseen type bound \eqref{ineq:BEintro} leads to the results \eqref{intro:bootstappr1}, \eqref{intro:bootstapprw} about bootstrap. In the framework of the Efron's bootstrapping scheme, condition \eqref{cond:XiYi} is modified with concentration bounds for the higher-order bootstrap moments (equal to the empirical moments) $\Eb ({\Xs_{i}}^{k})= n^{-1}\sum_{i=1}^{n}(X_{i}-\bar{X})^{k}$ for $k=2,\dots,K-1$, where $\Eb(\cdot)\eqdef\E(\cdot\cond X_{1},\dots,X_{n})$. For the case of the weighted bootstrap, condition \eqref{def:wb_eps_intro}  implies $\Eb ({\Xb_{i}}^{k}) =X_{i}^{k}\E(\veps_{i}^{k})=X_{i}^{k}$, $k=2,3$. In this way, the concentration properties of the empirical moments around the theoretical ones together with  the higher-order Berry--Esseen bounds of the form \eqref{ineq:BEintro} determine accuracy of the bootstrap procedures.
 Let us emphasize that the considered higher-order approximations play a key role for obtaining the improved accuracy of bootstrap procedures in terms of the ratio of $\dimp$ and $n$. For example, consider the weighted bootstrap procedure with a simplified condition on the random weights $\veps_{i}$. If $\E\veps_{i}=0$ and $\E(\veps_{i}^{2})=1$, then
 \begin{gather} 
\label{ineq:sumsmom12}
\E S_{n} =\E \Sb_{n}, \quad \E (S_{n}S_{n}^{\T}) =\E (\Sb_{n}{\Sb_{n}}^{\T}).
 \end{gather}
Using \eqref{ineq:sumsmom12} and a normal approximation between probability distributions of $\|S_{n}\|$ and $\|\Sb_{n}\|$ (e.g. the results of \cite{Bentkus2003BE}, or the inequalities by \cite{SpZh2014PMB} for the non-i.i.d. case), one can obtain an approximation bound similar to \eqref{intro:bootstapprw}, with an error term $\left({C_{\Sigma}^{3}
{\E\left(\|X_{1}\|^{3}\right)}
 }/{\sqrt{n}}\right)^{1/4}$ which is less sharp than \eqref{intro:bootstapprw} in the ratio between $\dimp$ and $n$.
 Using also the condition $\E(\veps_{i}^{3})=1$, we obtain
\begin{gather}
  \forall \alpha\in \R^{\dimp}\  \E \{(\alpha^{\T} S_{n})^{3}\} =\E \{(\alpha^{\T}\Sb_{n})^{3}\},
  \label{ineq:sumsmom3}
 \end{gather}
 and this property leads to the improved error term in \eqref{intro:bootstapprw}.   In order to employ the information about the third moments, as in  \eqref{ineq:sumsmom3}, one needs to use an approximation scheme that is more general than the normal approximation. For this purpose we establish the multivariate higher-order Berry--Esseen inequalities (Section \ref{SECT:MAINRES_BE}).
  
The methods introduced in the paper allow to consider an important and a more general model, namely, the Smooth Function Model introduced by 
\cite{Bhattacharya1978validity} and \cite{Hall1992bootstbook} (Chapter 2.4). In this model, the object of interest is $f(\mu)$, where $f:\R^{\dimp}\mapsto \R$ is a smooth function and $\mu$ is an unknown expected value if $X_{i}$. The bootstrap estimators allow to approximate $f(\Xbar)-f(\mu)$ in distribution, and,  therefore, to establish a confidence set for $f(\mu)$. This also includes the case, when one aims at constructing a confidence set for $\mu$ in the form $f(\Xbar-\mu)$. In Section \ref{SECT:MAINRES_BOOTSTR} we establish the approximation bounds similar to  \eqref{intro:bootstappr1} and  \eqref{intro:bootstapprw} for the Smooth Function Model. 

  The weighted or the multiplier bootstrap procedure is useful in the situations, 
  when it is required to resample a solution of estimating equations,   or a maximum likelihood estimator, or in the case when the random summands $\{X_{i}\}_{i=1}^{n}$ are not necessarily identically distributed (see, e.g., \cite{Mammen1993bootstrap,ChattBose2005generalized}). 
 The present results for the weighted bootstrap lead also to an improvement of accuracy of a weighted bootstrap procedure for general log-likelihood ratio statistics under possible model misspecification.  \cite{SpZh2014PMB} considered the weighted bootstrap for estimation of quantiles of a log-likelihood ratio, they showed that if a  parametric model is not severely misspecified, then 
the accuracy of bootstrap log-likelihood ratio quantiles corresponds to the accuracy of the normal approximation between statistics of the type $\|S_{n}\|$ and  $\|\Sb_{n}\|$.
Using inequality \eqref{intro:bootstapprw}, we infer that the accuracy of the weighted bootstrap for log-likelihood ratio depends rather on accuracy of the Wilks-type bounds, than on the normal approximation. We employ this result for construction of likelihood-based confidence sets.

Below we give an overview of the existing literature about bootstrap accuracy. Resampling methods are widely used for statistical inference in various applications. The bootstrap is well-known for its good performance in the situations when the amount of data is small (see, e.g. \cite{Horowitz2001bootstrapHandbook}), however, there are relatively few results about accuracy of the bootstrap in a non-asymptotic set-up. Most of the existing results are quite recent. \cite{ArlotBlanch2010} studied generalized weighted bootstrap for construction of non-asymptotic confidence bounds in $\ell_{r}$ -norm ($r \in [1, \infty] $) for the mean value of high-dimensional random vectors  with a symmetric and bounded (or with the normal) distribution. \cite{ChernoMultBoot} 
established Gaussian approximation results, as well as accuracy of the multiplier and the Efron's bootstrap for maxima of sums of high-dimensional vectors in a very general set-up. \cite{Chernozhukov.et.al.(2014a)}  extended the results from maxima to general hyperrectangles and sparsely convex sets.  The results of \cite{ChernoMultBoot,Chernozhukov.et.al.(2014a)} 
 allow the dimension $\dimp$ grow as $O(\exp(C n^{c}))$ for some constants $c,C>0$.
\cite{SpZh2014PMB} considered the multiplier bootstrap for estimation of quantiles of a general log-likelihood ratio under model misspecification. \cite{Zh2015UMB} extended this methodology for the simultaneous likelihood-based inference in the case of exponentially large number of models.

In the asymptotic high-dimensional setting when both the parameter dimension $\dimp$ and the sample size $n$ are large,  \cite{BickelFreedman1983bootstrapping,Mammen1989asymptotics,Mammen1993bootstrap} studied accuracy of the Efron's and the wild bootstrap for the linear regression model and for M-estimators; \cite{ChattBose2005generalized} studied generalized bootstrap for estimating equations also in high-dimensional asymptotic framework.
 \cite{Mammen1993bootstrap} studied validity and higher-order accuracy of the wild bootstrap (or Wu's bootstrap, first proposed by \cite{Wu1986wildboot}) under the condition  $\E(\veps_{i}^{3})=1$ on the weights, in context of linear contrasts in high dimensional linear models and for bootstrapping F-tests. \cite{Liu1988bootstr} used the condition $\E(\veps_{i}^{3})=1$ in order to obtain  the second order accuracy of the wild bootstrap.

One of the basic ways of studying the properties of bootstrap procedures is to consider asymptotic approximations of distributions of an initial statistic and its bootstrap estimate, e.g. using central limit theorems or their refinements with Edgeworth expansions (see \cite{Praestgaard1990bootstrap,PraestgaardWellner1993exchangeably}, \cite{Hall1992bootstbook,Mammen1992does,Barbe1995weighted,Shao1995jackknife,VaartWellner1996weak,Janssen2003bootstrap}, and references therein). Berry--Esseen type inequalities had been first used by  \cite{Singh1981asymptotic} and \cite{Liu1988bootstr} in the framework of bootstrap. \cite{HolmesReinert2004Stein} established bootstrap consistency in various settings using Stein's method. 

Below we discuss the literature about Berry--Esseen type bounds.
The problem of approximation of a probability distribution of the sum  $S_{n}$ belongs to the class of Central Limit Problems which has a long history of studies, see the paper by \cite{Loeve1950fundamental} for a detailed overview. 
\cite{Ibragimov1966accuracy} studied convergence of a distribution of $S_{n}$ in case of i.i.d. scalar summands,  to the standard normal law, under the higher moments condition; the author obtained a higher-order accuracy using Edgeworth expansion.  \cite{Zolotarev1965closeness} introduced pseudomoments, which characterize closeness of moments of two distributions, for estimation of convergence rates in limit theorems; such limit theorems are called nonclassical.   In the multivariate case, some of the first nonclassical results about normal approximation on closed convex sets had been obtained by  \cite{Paulauskas1975estimate,Rotar1978non} and \cite{Ul1979more}.   
To the best of our knowledge, the problem of approximation of a probability distribution of $S_{n}$ under the higher moments condition \eqref{cond:XiYi} and with an explicit dependence on the dimension $\dimp$, had not been studied before.

Now let us summarize the contribution of this paper to the existing literature.  In order to study the properties in a high-dimensional non-asymptotic setting, one needs to use new approaches and techniques. 
The methodology developed in this paper allows to consider higher-order properties of bootstrap methods in the modern set-up. To the best of our knowledge this had not been done in the earlier literature. The main theoretical tools, namely, the multivariate higher-order Berry--Esseen inequalities might be interesting by themselves. 
The approximation bounds established here allow to track the dependence of the error terms on the dimension, on the sample size, and on the moments of the considered  distributions. We provide examples, showing that the obtained error rates cannot be improved under the considered conditions.  
In addition, we refined an accuracy of the weighted/multiplier bootstrap procedure for the general log-likelihood ratio statistics. 

\subsection*{Structure of the paper} 
The results about accuracy of bootstrap rely on Berry-Essen type inequalities, for this reason we first present the latter results 
in Sections \ref{SECT:MAINRES_BE} and \ref{section:approxdistribution}. 
Section \ref{SECT:MAINRES_BOOTSTR}  contains theoretical results about accuracy of the bootstrap. Sections A  and  B in the supplement \citep{Zhilova2016supp} contain proofs of the statements from Sections  \ref{SECT:MAINRES_BE} and \ref{SECT:MAINRES_BOOTSTR} respectively. 
 Section \ref{sect:numer} presents results of numerical experiments.

\subsection*{Notation}
\(\|\cdot\|\) denotes the Euclidean norm for vectors and the operator norm for matrices or tensors; 
\(\SPDp\) denotes the set of symmetric positive definite real-valued matrices of size \(\dimp\times\dimp\); $\BCl$ is the set of all closed Euclidean balls in $\R^{\dimp}$; 
\(\Id_{\dimp}\) is the identity matrix of size \(\dimp\times\dimp\);  if $X$ is a vector in $\R^{\dimp}$, \(X^{k}\) stands for the tensor power  \(X^{\otimes k}\); for $f:\R^{\dimp}\mapsto \R$ and $h\in \R^{\dimp}$,  $f^{(s)}(x)h^{s}$ denotes the higher-order directional derivative $(h^{\T}\nabla)^{s}f(x)$; 
$C$ indicates a positive generic constant unless specified otherwise.
%
%

\section{Higher-order Berry--Esseen inequalities}
\label{SECT:MAINRES_BE}
Consider independent random vectors \(X_{1},\dots,X_{n}\in \R^{\dimp}\) such that \(\forall i=1,\dots,n\) \(\E X_{i}=0\), \(\Var(X_{i})\in\SPDp\), \(\E( \|X_{i}\|^{K})<\infty\) for some integer \(K\geq 3\) . Let $Y_1,\dots, Y_n\in\R^{\dimp}$ be independent random vectors, and such that \(\forall i=1,\dots,n\)  
\begin{equation}
\label{eq:mom_2}
\begin{gathered}
\text{$Y_i$ is  independent of $X_1,\dots,X_n$},\  \E(\|Y_i\|^K)<\infty,\\
\E (X_{i}^{k})=\E (Y_{i}^{k})\ \forall k=1,\dots,K-1.
\end{gathered}
\end{equation}
A formal definition of the equality of the higher-order moments of vector-valued random variables (as in  \eqref{eq:mom_2}) is given in \eqref{eq:mom}.  
We assume also that \(\forall i=1,\dots,n\)
\begin{equation}
\label{cond:XiYi2}
\begin{gathered}
\text{\(\exists\) independent r.v. \(\Zy_{i},\Uy_{i}\in\R^{\dimp}\) s.t. }
 Y_{i}\overset{d}{=}\Zy_{i}+\Uy_{i},\\
 \E \Zy_{i}= \E \Uy_{i}=0, 
\text{\(\Zy_{i}\sim \mathcal{N}(0,\Sigmazi)\) for some \(\Sigmazi\in \SPDp\).}
\end{gathered}
\end{equation}
Consider the following sums of mutually independent random vectors with zero mean:
\begin{gather}
\label{def:sums_mainres}
S_{n}\eqdef n^{-1/2}{{\textstyle \sum\nolimits_{i=1}^{n}}} X_{i},\quad
\St_{n}\eqdef n^{-1/2}{{\textstyle \sum\nolimits_{i=1}^{n}}} Y_{i}.
\end{gather}
We establish uniform approximation bounds between probability distributions of $S_{n}$ and $\St_{n}$ on the set \(\BCl\) of all Euclidean balls in $\R^{\dimp}$. 
Theorems \ref{theor:NcBE} and \ref{theor:NcBE_noniid} treat the cases when \(\{X_{i}\}_{i=1}^{n}\) are i.i.d. and independent but not necessarily identically distributed vectors correspondingly.  For the case of i.i.d. summands $X_{i}$ (and, hence, i.i.d. $\Zy_{i}$) denote 
 \begin{equation}
\label{def:Cz_mainres}
 \Sigmaz\eqdef\Sigmazi= \Var(\Zy_{i}), \quad \Cz \eqdef \|\Sigmaz^{-1/2}\|. 
\end{equation}

\begin{theorem}
\label{theor:NcBE}
Consider the random vectors \(\{X_{i}\}_{i=1}^{n}\) introduced above, suppose that they are i.i.d., and that there exist i.i.d. approximating random vectors \(\{Y_{i}\}_{i=1}^{n}\)  meeting conditions   \eqref{eq:mom_2}  and \eqref{cond:XiYi2}. It holds for the sums $S_n$ and $\St_n$ defined in \eqref{def:sums_mainres} 
\begin{align*}
\sup_{B\in \BCl}\left|
\P\bigl(S_{n}\in B \right)
-
\P\bigl(\tilde{S}_{n}\in B \bigl)
\bigl|
&\leq  \CBEBiid\frac{ \left\{\Cz^{K}\E\left(\|X_{1}\|^{K}+\|Y_{1}\|^{K}\right)\right\}^{{1}/{(K-2)}}}{n^{1/2}},
\end{align*}
where constant $\CBEBiid>0$ depends only on $K$; a detailed definition of $\CBEBiid$  is given in the proof (see (A.52) in Section A.2 of the supplement \citep{Zhilova2016supp}).
\end{theorem}
\begin{remark}[The case of the normal approximation]
\label{rem:Gaussian}
If the approximating random vectors  \(Y_i\) are normally distributed, then \(\Uy_{i}\equiv 0\), \(Y_{i}\sim\mathcal{N}(0,\Var(X_{i}))\),  \(\Sigmaz=\Var(X_{i})\), and $\Cz= \|\{\Var(X_{i})\}^{-1/2}\|$. 
Furthermore, if $K=3$ and $Y_{i}$ are standard normal, then the bound in Theorem \ref{theor:NcBE} is similar to the classical multivariate Berry--Esseen inequality by \cite{Bentkus2003BE}. If $K>3$ and $Y_{i}$  are normally distributed, the term $\|X_{1}\|$ enters the bound above with a better power, than in the classical case where $K=3$. In this way, Theorem \ref{theor:NcBE} extends the classical normal approximation result.
\end{remark}
\begin{remark}[Dependence on $\Cz$]
\label{remark:Czdepend}
The approximation bound in Theorem \ref{theor:NcBE} depends on  
\(\Cz=\|\Sigmaz^{-1/2}\|\), where \(\Sigmaz\) is a covariance matrix of the normal component \(\Zy_{i}\) of the approximating distribution \(Y_{i}\).   
In Lemma \ref{lemma:Yexist} (Section \ref{section:approxdistribution}) we show that if a cardinality of a support of  \(X_{i}\) is sufficiently large, then there exist random vectors $\Uy_{i}$ such that  \(\Sigmaz\) is positive definite. Therefore, it holds $ \lambda_{\Sigma}^{-1} \leq \Cz<\infty$, where  $\lambda_{\Sigma}^{2}$ is the smallest eigenvalue of $\Var X_{i}$. In Lemma \ref{lemma:Yexist3} we consider the case when the number of coinciding moments between $X_{i}$ and $Y_{i}$ is $K-1=3$; we show that for any $c_{0}\in (0,\lambda_{\Sigma})$, there exists distribution $Y_{i}=\Zy_{i}+\Uy_{i}$ such that $\Cz< c_{0}^{-1}$. Hence $\Cz$ can be taken as a generic constant for $K=4$.   
Moreover, if the coordinates of the vector $X_{i}$ are mutually independent, then the problem of characterizing $\Sigmaz$ and $\Cz$ becomes one-dimensional and, therefore,  $\Cz$ does not depend on $\dimp$ in this case.
\end{remark}

\begin{remark}[Accuracy of the approximation]
The error term in Theorem \ref{theor:NcBE} is proportional to $ \bigl\{\dimp n^{-(K-2)/K})\bigr\}^{1/(K-2)}$ if $\|X_{1}\|, \|Y_{1}\|\leq  \sqrt{\dimp}$ a.s.  
In Lemma \ref{lemma:BElowerbound} below, we show that for $K\geq 3$ condition  $\dimp=o(n^{(K-2)/K})$ as $n\to \infty$ is necessary for $\sup_{x\in\R}|\P(\|S_{n}\|\leq x) -\P(\|\tilde{S}_{n}\|\leq x) | \to 0$, $n\to \infty$ under the conditions of Theorem \ref{theor:NcBE}.
\end{remark}

\begin{lemma}[Necessity of the condition $\dimp=o(n^{(K-2)/K})$]
\label{lemma:BElowerbound}
Let $\{X_{i}\}_{i=1}^{n}$ be random vectors as in Theorem \ref{theor:NcBE}.  Suppose that $\E (\|X_{i}\|^{K+2})<\infty$ for an integer $K>3$. There exist random vectors $\{Y_{i}\}_{i=1}^{n}$ satisfying conditions of Theorem \ref{theor:NcBE}, and s.t. that  the condition $\dimp=o(n^{(K-2)/K})$ for $n\to \infty$ is necessary for  $\sup_{x\in\R}|\P(\|S_{n}\|\leq x) -\P(\|\tilde{S}_{n}\|\leq x) | \to 0$ as $n\to \infty$.
\end{lemma}
\begin{remark}
In the recent paper \cite{Zhai2016multivariate} considers a multivariate CLT in $\mathcal{W}_2$-distance. The author shows that if $X_{1},\dots, X_{n}\in \R^{\dimp}$ are i.i.d. with mean zero and such that $\|X_{i}\|\leq \beta$ a.s. for some constant $\beta>0$, then
\begin{align}
\label{ineq:w2}
\mathcal{W}_{2}(S_{n},Z) \leq {5\sqrt{\dimp}\beta (1+\log n)}/{\sqrt{n}},
\end{align}
 where $Z\sim \mathcal{N}(0,\Var X_{1})$, and $\mathcal{W}_2$ is the 2-Wasserstein distance. This result implies, that if $\beta\leq C\sqrt{\dimp}$, then $\mathcal{W}_{2}(S_{n},Z) \leq C \dimp (1+ \log n)/\sqrt{n}$. 
  In Lemma \ref{lemma:comparwithW2} below, we consider a uniform bound on $\sup\nolimits_{B\in \BCl}\left|
\P(S_{n}\in B)-\P(Z\in B)
\right|$, using the result \eqref{ineq:w2}. It turns out that under conditions \eqref{eq:mom_2}, \eqref{cond:XiYi2} for $K\geq 4$, the higher-order Berry--Esseen type inequality in Theorem \ref{theor:NcBE} yields a better accuracy w.r.t. the sample size $n$, and w.r.t. the ratio between $\dimp$ and $n$. Indeed, inequality \eqref{ineq:W2_2} below, which follows from the results of \cite{Zhai2016multivariate}, has an error term of order $(\dimp^{2}/n)^{1/3}$ (up to $\log n$). Whereas Theorem \ref{theor:NcBE} (for the case $Y_{i}\sim \mathcal{N}(0,\Var{X_{i}})$) provides a smaller error term of order $(\dimp^{K/(K-2)}/n)^{1/2}$ for $K\geq 4$.
\begin{lemma}
\label{lemma:comparwithW2}
Let $X_{1},\dots, X_{n}\in \R^{\dimp}$ be i.i.d. random vectors, such that $\Var X_{i}=\Id_{\dimp}$ and $\|X_{i}\|\leq \beta$ a.s. for some constant $\beta>0$. Let $Z\sim \mathcal{N}(0,\Id_{\dimp})$. The results of \cite{Zhai2016multivariate} imply 
\begin{align}
\nonumber
\sup\nolimits_{B\in \BCl}\left|
\P(S_{n}\in B)-\P(Z\in B)
\right|&\leq {C_{1} \dimp^{1/3} \beta^{2/3}(1+\log n)}/{n^{1/3}}
\\
&\leq {C_{2} \dimp^{2/3}(1+\log n)}/{n^{1/3}},
\label{ineq:W2_2}
\end{align}
where in the latter inequality one takes $\beta\leq c\sqrt{\dimp}$. 
Here $c, C_{1}$, and $C_{2}$ denote positive generic constants.
\end{lemma}
\end{remark}
%

Now let us consider the case when the random summands \(X_{i}\)  are independent but not necessarily identically distributed (non-i.i.d.).  
Denote  
\begin{equation}
\label{def:Sigmazi}
 \SigmazL\eqdef n^{-1}{{\textstyle\sum\nolimits_{i=1}^{n}}}\Sigmazi, \quad \CzL\eqdef \|\SigmazL^{-1/2}\|. 
\end{equation}

\begin{theorem}\label{theor:NcBE_noniid}
Consider random vectors \(\{X_{i}\}_{i=1}^{n}\) introduced above, suppose that they are independent but not necessarily identically distributed, and that there exist independent approximating vectors \(\{Y_{i}\}_{i=1}^{n}\)  meeting conditions   \eqref{eq:mom_2}  and \eqref{cond:XiYi2}. It holds for the sums $S_n$ and $\St_n$ defined in \eqref{def:sums_mainres} 
 \begin{align*}
 \sup_{B\in \BCl}\bigl|
\P\left(S_{n}\in B \right)
-
\P\bigl(\tilde{S}_{n}\in B \bigl)
\bigr|
&\leq 
\CBEBniid
\bigl\{
{ \CzL^{K}{{\textstyle\sum\nolimits_{i=1}^{n}}}
\E\left(\|X_{i}\|^{K}+\|Y_{i}\|^{K}\right)
}{n^{-K/2}}
\bigr\}^{\frac{1}{K+1}},
 \end{align*}
 where 
 constant $\CBEBniid>0$ depends only on $K$, it is defined in the proof (see (A.57) in Section A.3 of the supplement \citep{Zhilova2016supp}).
\end{theorem}
 \begin{remark}
The proof of Theorem \ref{theor:NcBE} largely exploits the assumption that the summands $\{X_{i}\}_{i=1}^{n}$ are identically distributed, and it does not directly apply to the non-i.i.d. case. This causes a difference between the error terms in Theorems \ref{theor:NcBE_noniid} and  \ref{theor:NcBE}, however, the critical ratio of $\dimp$ and $n$ (namely, $\dimp ^{K/(K-2)}/n$) in the error terms remains the same in both results.
We leave an improvement of the power $^{1/(K+1)}$ in the non-i.i.d. case for the future work.
 \end{remark}
 \begin{remark}
\label{remark:classes}
The proofs of the Berry--Esseen type inequalities (Theorems \ref{theor:NcBE} and \ref{theor:NcBE_noniid}) exploit the following properties of the set $\BCl$ (cf. \cite{Bentkus2003BE,Bentkus2005Lyapunov} and \cite{ChernoMultBoot,Chernozhukov.et.al.(2014a)}):
 \begin{itemize}[leftmargin=*]
\item Invariance under rescaling and under taking shifts, i.e. if $B\in \BCl$, then $\forall x\in \R^{\dimp}$ and $\forall a\in\R, a>0$ it holds $aB+x \in \BCl$. 
\item  Invariance under taking $\varepsilon$-neighborhood w.r.t. the $\ell_{2}$-norm: 
consider an arbitrary $B\in \BCl$, \(B=B_{r}(x_{0})\eqdef \left\{x \in \R^{\dimp}: \|x-x_{0}\|\leq r\right\}\), then the $\varepsilon$-neighborhood of $B$ reads as $B^{\epsc}=
 B_{r+\epsc}(x_{0}), \text{ if }r+\epsc\geq 0$, and $B^{\epsc}=\emptyset$ otherwise. 
Therefore $B^{\epsc}\in \BCl$.
\end{itemize}
The same properties hold for the set $\HCl$ of all half-spaces in $\R^{\dimp}$. In proposition A.1 (Section A.3 of the supplement \citep{Zhilova2016supp}), we consider a higher-order Berry--Esseen approximation for $S_{n}$ uniformly over the set $\HCl$, similarly to Theorems \ref{theor:NcBE} and \ref{theor:NcBE_noniid}.
\end{remark}
Corollary \ref{theor:triangle} below follows directly from the previous theorems and the  triangle inequality. It  justifies a higher-order accuracy of approximation between two probability distributions with matching moments.
 \begin{corollaryT}
\label{theor:triangle}
Consider random vectors \(\{X_{i}\}_{i=1}^{n}\) introduced above, and suppose that there exist independent approximating vectors \(\{Y_{i}\}_{i=1}^{n}\)  meeting conditions   \eqref{eq:mom_2}  and \eqref{cond:XiYi2}. Consider also independent random vectors  \(X_{1}^{\prime},\dots,X_{n}^{\prime}\in \R^{\dimp}\), that are independent of $\{X_{i}\}_{i=1}^{n}$,  $\{Y_{i}\}_{i=1}^{n}$, and such that $\forall i=1,\dots, n$
$\E \left(\|X^{\prime}_i\|^K\right)<\infty,$ 
$\E \bigl(X_{i}^{k}\bigr)=\E \bigl({X_{i}^{\prime}}^{k}\bigr)\ \forall k=1,\dots,K-1.$
Let also 
$
S^{\prime}_{n}\eqdef n^{-1/2}\sum\nolimits_{i=1}^{n}X^{\prime}_{i},$ and 
$ 
\Delta_{n}^{\prime}\eqdef
\sup\nolimits_{B\in \BCl}\left|
\P\left( S_{n}\in B \right)
-
\P\bigl(S^{\prime}_{n}\in B \bigl)
\right|.
$
\begin{enumerate}
\item 
If conditions of Theorem \ref{theor:NcBE} are fulfilled and \(\{X_{i}^{\prime}\}_{i=1}^{n}\) are i.i.d.
Let also ${L}_{K}\eqdef 
\E\left(\|X_{1}\|^{K}+\|Y_{1}\|^{K}\right)$ and ${L}_{K}^{\prime}\eqdef 
\E\left(\|X_{1}^{\prime}\|^{K}+\|Y_{1}\|^{K}\right)$, then
\begin{align*}
\Delta_{n}^{\prime} &\leq 
 \CBEBiid\Cz^{K/(K-2)} 
 { 
 \bigl\{
 L_{K}^{1/(K-2)}+ L_{K}^{\prime\,1/(K-2)}
 \bigr\}
 }{n^{-1/2}}.
\end{align*}
\item
If conditions of Theorem \ref{theor:NcBE_noniid} are fulfilled and \(\{X_{i}^{\prime}\}_{i=1}^{n}\) are not necessarily identically distributed. Let also $\bar{L}_{K}\eqdef n^{-1}\sum\nolimits_{i=1}^{n}
\E(\|X_{i}\|^{K}+\|Y_{i}\|^{K})$ and $\bar{L}_{K}^{\prime}\eqdef  n^{-1}\sum\nolimits_{i=1}^{n}
\E(\|X_{i}^{\prime}\|^{K}+\|Y_{i}\|^{K})$, then
\begin{align*}
\Delta_{n}^{\prime} &\leq 
\CBEBniid \CzL^{{K}/{(K+1)}}
{\bigl\{ \bar{L}_{K}^{1/(K+1)}+{\bar{L}_{K}}^{\prime\,1/(K+1)}\bigr\}}{n^{-0.5(K-2)/(K+1)}}.
\end{align*}
\end{enumerate}
\end{corollaryT}

\section{Properties of the approximating distribution}
\label{section:approxdistribution}
\begin{lemma}[Existence of the approximating distribution]
\label{lemma:Yexist}
Let  a random vector \(X\) be supported in a closed set $A\subseteq \R^{\dimp}$, and let $X$ be such that \(\E X=0\), \(\Var X \in \SPDp\), \(\E(\|X\|^{K+2})<\infty\) 
for some integer \(K\geq 2\).  
If $X$ is continuously distributed, there exists a random vector \(Y\eqdef \Zy+\Uy\),  such that  \(\E(\|Y\|^{K+2})<\infty\), where \(\Zy,\Uy\in\R{^\dimp}\) are independent, \(\Zy\sim \mathcal{N}(0,\Sigmaz)\) for some \(\Sigmaz\in \SPDp\), and \(\E(X^{k})=\E(Y^{k})\) for all \(k=0,\dots,K\). 
 Furthermore, if $X$ is a sub-Gaussian random vector, then there exists a sub-Gaussian approximating distribution $Y$ satisfying the above conditions. If $X$ has a discrete probability distribution supported on $M$ points in $\R^{\dimp}$ such that each coordinate of $X$ is supported on at least $m$ points in $\R$, then the lemma's statement holds for $X$ when $M\geq 1+(K+2) m^{\dimp-1}$.
\end{lemma}
\begin{proof}[Proof of Lemma \ref{lemma:Yexist}]
Denote
$m_{k}\eqdef \E(X^{k})$, and $u_{k}\eqdef \E({\Uy}^{k})$ for $k=0,1,2,\dots$
Conditioning on \(\Uy\) leads to \(\LL(Y\cond\Uy)= \mathcal{N}(\Uy,\Sigmaz)\) and to the following system of linear equations:
\begin{EQA}[ccccll]
m_{0}&=&\E (\Zy+\Uy)^{0}&=&u_{0},\ &\hspace{-7.5cm}m_{2}=\E (\Zy+\Uy)^{2}=u_{2}+\Sigmaz,\\
m_{1}&=&\E (\Zy+\Uy)&=&u_{1},\ &\hspace{-7.5cm}m_{3}=\E (\Zy+\Uy)^{3}=u_{3},\\
m_{K}&=&\E (\Zy+\Uy)^{K}&=&
K!{{\textstyle\sum\nolimits_{l=0}^{[K/2]}}}S_{\dimp\hspace{-0.05cm}\Ind_K}{u_{K-2l}\otimes\vect(\Sigmaz)^{l}}\{l!(K-2l)! 2^{l}\}^{-1},
\end{EQA}
where \(S_{\dimp\hspace{-0.05cm}\Ind_K}\) is the symmetrizer operator acting on the $K$-th tensor power of $\R^{\dimp}$;  this formula for the raw moments of the multivariate normal distribution is  given in the work \cite{Holmquist1988moments}. The solution \(\{u_{k}(\Sigmaz)\}_{k=0}^{K}\) of this system depends on \(\Sigmaz\) continuously. Moreover,
\begin{equation}
\label{prop:sigma1}
\text{if \(\Sigmaz=0\), then \(u_{k}(\Sigmaz)=m_{k}\) \(\forall\, k=0,\dots,K.\)}
\end{equation}
In order to prove the lemma's statement, it is sufficient to show that there exists \(\Sigmaz\in \SPDp\), s.t. the solution \(\{u_{k}(\Sigmaz)\}_{k=0}^{K}\) also solves the following multivariate truncated moment problem:
\begin{itemize}[leftmargin=*]
\item[]
given a \(\dimp\)-dimensional real multisequence \(\{u_{k}(\Sigmaz)\}_{k=0}^{K}\), does there exist a positive Borel measure \(\mu\) s.t. $\supp \mu\subseteq A$ and \(\int_{\R^{\dimp}}x^{k}d\mu(x)=u_{k}(\Sigmaz)\ \forall\, k=0,\dots,K\)?
\end{itemize}

The work of \cite{Curto2008analogue} provides necessary and sufficient conditions for solvability of multivariate truncated moment problems. Before stating these conditions we introduce some notation. 
Let \(\mathcal{P}_{K}\) denote the space of polynomials: $\R^{\dimp}\mapsto \R$, of degree \(\leq K\), and with real coefficients. A polynomial \(p=p(x)=\sum_{|\iv|\leq K} a_{\iv}x^{\iv} \in \mathcal{P}_{K} \) is positive (or strictly positive) on $A$, if \(p(x)\geq 0\) (or $p(x)>0$) for all \(x\in A\). Here \(\iv\eqdef(i_{1},\dots,i_{\dimp})\in \mathbb{N}^{\dimp}_{0}\) denotes multi-index, \(|\iv|=\sum_{j=1}^{\dimp}i_{j}\), and \(x^{\iv}\eqdef x_{1}^{i_{1}}\dots x_{\dimp}^{i_{\dimp}}\). For a multisequence \(\{u_{\iv}\}_{|\iv|\leq K}\) the Riesz functional \(L: \mathcal{P}_{K}\mapsto \R\) is defined as \( L(\sum_{|\iv|\leq K} a_{\iv}x^{\iv})\eqdef\sum_{|\iv|\leq K}  a_{\iv}u_{\iv}\). 
If the truncated moment problem is soluble, we can write
\begin{equation}
\label{eq:RHthmExpect}
{{\textstyle L(p)=\sum_{|\iv|\leq K}\nolimits  a_{\iv}u_{\iv}={{\int_{\R^{\dimp}}}}p(x)d \mu(x).}}
\end{equation}
\cite{Curto2008analogue} showed that a multisequence \(\{u_{\iv}\}_{|\iv|\leq K}\) solves the multivariate  truncated moment problem on the set $A$ iff there exists an extension \(\{\tilde{u}_{\iv}\}_{|\iv|\leq K+2}\) of \(\{u_{\iv}\}_{|\iv|\leq K}\)  (i.e. \(\tilde{u}_{\iv} = u_{\iv}\) for all \(\iv: |\iv|\leq K\)), such that for the corresponding Riesz functional  \( \tilde{L}(\sum_{|\iv|\leq K+2} a_{\iv}x^{\iv})\eqdef\sum_{|\iv|\leq K+2} a_{\iv}\tilde{u}_{\iv}\) it holds:
\begin{equation}
\label{eq:RHtheorem}
{{\textstyle \text{if \(p\in \mathcal{P}_{K+2}\) and \(p\) is positive on $A$, then \(\tilde{L}(p)\geq 0.\)}}}
\end{equation}

Firstly, consider the case of a continuously distributed $X$. Due to the definitions of $m_{k}$ and $u_{k}$, and by the theorem of \cite{Curto2008analogue} there exists an extension \(\{{m}_{k}\}_{k=0}^{K+2}\), s.t. its corresponding Riesz functional \( \tilde{L}_{m}(\sum_{|\iv|\leq K+2} a_{\iv}x^{\iv})\eqdef\sum a_{\iv}{m}_{\iv}\) satisfies \eqref{eq:RHtheorem}. Moreover, if $p\in \mathcal{P}_{K+2}$ is s.t. \(p(x)>0\) $\forall x\in A$, then \(\tilde{L}(p)> 0\); if \(\tilde{L}(p)= 0\) for some  $p\in \mathcal{P}_{K+2}$  nonnegative on $A$, then $\P\left(\forall\,x\in A\ p(x)=0\right)=1$. The extension 
\(\{{m}_{k}\}_{k=0}^{K+2}\)  leads to the extended sequence  \(\{\tilde{u}_{k}(\Sigmaz)\}_{k=0}^{K+2}\).  Property \eqref{prop:sigma1}, continuity of the solutions  \(\{\tilde{u}_{k}(\Sigmaz)\}_{k=0}^{K+2}\) w.r.t. \(\Sigmaz\), and \eqref{eq:RHthmExpect} imply that there exists some \(\Sigmaz \in \SPDp\) s.t.
 the corresponding Riesz functional \(\tilde{L}_{u}(p)\eqdef \sum_{|\iv|\leq K+2} a_{\iv}\tilde{u}_{\iv}(\Sigmaz)>0\)  for all \(p=\sum_{|\iv|\leq K+2} a_{\iv}x^{\iv}\in \mathcal{P}_{K+2}\) such that \(p>0\) on $A$ (here we use that $\SPDp$ is a dense subset of the set of symmetric positive semidefinite matrices in $\R^{p\times p}$; see e.g. Chapter 2.4 in \cite{Boyd2004convex}).  This finalizes the proof for the continuous case. 
  
  Now let $X$ have a discrete probability distribution. Let $m$ denote the minimal cardinality of the supports of $X$'s coordinates. Consider a polynomial $q\in \mathcal{P}_{K+2}$ of degree $K+2$, such that $q\not\equiv  0$ on the support  $A$. According to the Schwartz-Zippel Lemma by \cite{Schwartz1980fast} and \cite{Zippel1979probabilistic}, the number of zeros of $q$ is $\leq (K+2) m^{\dimp-1}$. Hence, if the set $A$ contains at least $1+ (K+2) m^{\dimp-1}$ points, then for any non-zero and positive polynomial $q$, \(\tilde{L}(q)> 0\), where $\tilde{L}$ is the Riesz functional, considered in the previous paragraph. This allows to apply here the arguments for the continuous case.
 
 If $X$ is sub-Gaussian, then $\left|m_{k}\gamma^{k}\right|^{1/k}\leq \left\{\E|\gamma^{\T}X|^{k}\right\}^{1/k}\leq C \sqrt{k} \ \forall k\geq 1, \ \forall \gamma\in \R^{\dimp}: \|\gamma\|=1$. Since the solutions ${u}_{k}(\Sigmaz)$ are linearly dependent on $\{m_{k}\}$, the sub-Gaussian property holds for the moments ${u}_{k}(\Sigmaz)$  as well.
\end{proof}

\begin{lemma}[Choice of $\Cz$ for 3 coinciding moments]
\label{lemma:Yexist3}
Let $X\in \R^{\dimp}$ be a random vector satisfying conditions of Lemma \ref{lemma:Yexist} for  $K=3$.  Let $\lambda_{\Sigma}$ denote the smallest eigenvalue of $\Sigma \eqdef \Var(X)$. Then for any $c_{0}\in(0, \lambda_{\Sigma})$ there exist independent random vectors \(\Zy,\Uy\in\R{^\dimp}\), such that for $Y{=}\Zy+\Uy$, it holds \(\E(\|Y\|^{4})<\infty\), \(\E(X^{k})=\E(Y^{k})\) for  \(k=1,2,3\), and \(\Zy\sim \mathcal{N}(0,\Sigmaz)\), where $\Sigmaz$ is some symmetric matrix with the smallest eigenvalue $> c_{0}$.
\end{lemma}
\begin{proof}[Proof of Lemma \ref{lemma:Yexist3}]
Consider  random variable $\varepsilon \in\R$ independent of $X $ and s.t. $\E \varepsilon=0, \E(\varepsilon^{2})= \alpha, \E(\varepsilon^{3})=1$ for some $\alpha\in(0,1)$. Such distribution exists by the criterion for solubility of the truncated Hamburger moment problem (see \cite{Curto1991recursiveness}). Indeed, if the Hankel matrix $\begin{psmallmatrix}1& 0\\ 0& \alpha  \end{psmallmatrix}$ has a positive Hankel  extension  $\begin{psmallmatrix}1& 0 & \alpha \\ 0& \alpha &1\\ \alpha&1 &\beta \end{psmallmatrix}$ for some  $\beta=\E(\varepsilon^{4})$, then there exists such random variable $\varepsilon$.

 Take $\tilde{X}\eqdef  X \varepsilon$, then $\E \tilde{X} = 0, \E (\tilde{X}^2)= \alpha \Sigma, \E (\tilde{X}^3)= \E ({X}^3).$ 
By Lemma \ref{lemma:Yexist} $\exists$ r.v. $\tilde{Y}=\tilde{Z}+\tilde{U} \in \R^{\dimp}$ s.t. $\tilde{Z}$ and $\tilde{U}$ are independent of each other, \(\tilde{Z}\sim \mathcal{N}(0,\tilde{\Sigma}_{\zv})\) for some $\tilde{\Sigma}_{\zv}\in \SPDp$, and  \(\E(\tilde{X}^{k})=\E(\tilde{Y}^{k})\) \(\forall\,k=0,1,2,3\). Let $\tilde{Z}_{1} \sim \mathcal{N}(0,  (1-\alpha)\Sigma)$ be independent of all random vectors considered in the proof, take
$Y\eqdef \tilde{Y}+\tilde{Z}_{1}=\tilde{Z}+\tilde{Z}_{1}+\tilde{U},$  
then it holds
$\E(Y)=0,  \E(Y^2)= \alpha \Sigma+  (1-\alpha) \Sigma=\Sigma,$ and 
$\E(Y^3)= \E(\tilde{Y}^3)= \E ({X}^3).$ The normal part of $Y$ is $Z\eqdef \tilde{Z}+\tilde{Z}_{1} \sim \mathcal{N}(0, (1-\alpha)\Sigma+\tilde{\Sigma}_{\zv})$. Let $\lambda_{\zv}$ denote the smallest eigenvalue of  $\Var Z$, then $ (1-\alpha)\lambda_{\Sigma}<\lambda_{\zv}<  \lambda_{\Sigma},$
where $\lambda_{\Sigma}>0$ is the smallest eigenvalue of $\Sigma$. Hence, taking $\alpha = 1-c_{0}/ \lambda_{\Sigma}$, we obtain the lemma's statement.
\end{proof}
%
%
\section{Validity and accuracy of the bootstrap procedures}
\label{SECT:MAINRES_BOOTSTR}
Here we study accuracy of the Efron's and the weighted bootstrap procedures in various settings.
We begin with the Efron's bootstrap in Section \ref{SECT:EB}; 
 Sections \ref{SECT:WB}, \ref{sect:wblikelihood} present the results for the weighted bootstrap. 

\subsection{Efron's bootstrap}
\label{SECT:EB}
Let $X_{1},\dots,X_{n}$ be i.i.d. random vectors with $\Sigma\eqdef \Var(X_{i})\in \SPDp$, let also $X_{i}$ be sub-Gaussian, i.e. it holds for some $\sigma^{2}>0$ and for all $\alpha \in \R^{\dimp}$
\begin{align}
\label{cond:subGauss1}
\E\bigl\{\exp(\alpha^{\T}X_{i})\bigr\}&\leq
\exp \left(\|\alpha\|^{2}\sigma^{2}/2\right).
\end{align}
Assume that there exist i.i.d. random vectors $Y_{1},\dots,Y_{n}$ satisfying \eqref{eq:mom_2}, \eqref{cond:XiYi2} for some integer $K\geq 3$. 
 Introduce resampled variables $\Xs_{1},\dots,\Xs_{n}$ with zero mean, according to the Efron's bootstrap methodology (\cite{Efron1979bootstrap,Efron1994introduction}): 
$\Pb(\Xs_{i}=X_{j}-\Xmean)=1/n\quad \forall i,j=1,\dots,n,$
 where $\Xmean=n^{-1}\sum_{i=1}^{n}X_{i}$, and  $\Pb(\cdot)=\P(\cdot\cond X_{1},\dots,X_{n})$. In this way, $\{\Xs_{j}\}_{j=1}^{n}$ are i.i.d., $\Eb(\Xs_{j})=0$, and 
$\Es({\Xs_{j}}^{k})=n^{-1}{{\textstyle\sum\nolimits_{i=1}^{n}}} (X_{i}- \Xmean)^{k}$, for $k\geq 1$.   The bootstrap approximation of the sum $S_{n}$ is  $\Ss_{n}\eqdef  n^{-1/2}{{\textstyle\sum\nolimits_{i=1}^{n}}}\Xs_{i}.$
Denote 
$\Cx\eqdef \|\Sigma^{-1/2}\|$. 
 By this definition, $\Cx< \Cz$. Assume also that a p.d.f. of $X$ is bounded with a constant $c_{f}>0$. 
  In the statements in Section \ref{SECT:MAINRES_BOOTSTR}, including the theorems below, we use notation from the previous Section \ref{SECT:MAINRES_BE}, e.g. constant $\CBEBiid$. 
  Let also  $ \tilde{C}_{x,k}\eqdef  \bigl(1+2\sqrt{x/\dimp}+2x/\dimp \bigr)^{k/2}$.
  
\begin{theorem}[Accuracy of the bootstrap for $S_{n}$ on the set $\BCl$]
\label{theorem:E_bootstr_l2_accur}
Suppose that the above conditions are fulfilled, 
then the following uniform bound holds on the set $\BCl$ of all Euclidean balls with probability $\geq 1-6\ex^{-x}$ for $x>0$:
\begin{align*}
&\hspace{-1.6cm} \sup\nolimits_{B\in \BCl}\left|
 \P\left(S_{n}\in B \right)
-
\Pb\bigl(\Ss_{n}\in B \bigl)
\right|
\\\leq \Deltas_{\BCl,iid}&\eqdef   \CBEBiid{ \left\{\Cz^{K}\E\left(\|X_{1}\|^{K}+\|Y_{1}\|^{K}\right)\right\}^{{1}/{(K-2)}}}{n^{-1/2}}
 \\&
 \quad+\CBEBiid (2\tilde{C}_{x,K})^{1/(K-2)}{(\Cz \sigma  \dimp^{1/2})^{K/(K-2)} }{n^{-1/2}}
 +R_{n,K},
\end{align*}
where 
$R_{n,K}\leq C\sqrt{{p}/{n}}  \Cz^{2}\sigma^{K+1}\Cx^{K-1} {C}_{x,K}$, and ${C}_{x,K}$ is defined in (B.27); 
a detailed definition of $R_{n,K}$ is given in (B.9), (B.10) (Section B.1 in the supplement \citep{Zhilova2016supp}).
\end{theorem}

\begin{remark}
\label{remark:ebootstr_ratio}
The error term $ \Deltas_{\BCl,iid}$ in the above result consists of two parts: one part corresponds to the higher-order Berry--Esseen type inequalities, another part $R_{n,K}$ comes from concentration bounds for higher-order empirical moments of $X_{i}$. 
If the ratio $\dimp^{K/(K-2)}/{n}$ is small, then the first part is small as well. 
Furthermore, $\forall K\geq 3$ $\dimp^{K/(K-2)}/{n} \geq \dimp/n$.  In Lemma \ref{lemma:lowerboundbootst} (in Section \ref{sect:someremarks}), we consider an example where the condition $\dimp/n= o(1)$ for $n\to \infty$ is required for the bootstrap consistency.
\end{remark}

In the following theorem we study accuracy of the Efron's bootstrap procedure for the  Smooth Function Model introduced by \cite{Bhattacharya1978validity} and \cite{Hall1992bootstbook} (Chapter 2.4). In this model the object of interest is $f(\mu)$, where $f:\R^{\dimp}\mapsto \R$ is a smooth function and $\mu$ is an unknown expected value if $X_{i}$. The bootstrap estimators allow to approximate $f(\Xbar)-f(\mu)$ in distribution, and,  therefore, to establish a confidence set for $f(\mu)$. This also includes the case, when we aim at constructing a confidence set for $\mu$ in the form $f(\Xbar-\mu)$. 
Consider i.i.d. $X_{1},\dots,X_{n}\in \R^{\dimp}$ with mean $\mu$ and sub-Gaussian tail behavior, i.e. condition \eqref{cond:subGauss1} holds for $X_{i}-\mu$. 
 Let $f:\R^{\dimp}\mapsto \R$ be at least twice continuously differentiable function, s.t.  $\forall h\in \R^{\dimp}$ $\sup_{x\in \R^{\dimp}}|f^{(2)}(x)h^{2}|\leq C_{f,2} \|h\|^{2}$ for some constant $C_{f,2}>0$. Assume also that  $f^{\prime}(\mu)\neq 0$, and $\|f^{\prime} (\mu)\|>C_{f,l} \sqrt{\dimp}$ for some constant $C_{f,l}>0$.
 Denote the resampled i.i.d. data $\Xs_{1},\dots,\Xs_{n}$ and the bootstrap empirical mean as follows:
\begin{gather*}
\Pb(\Xs_{i}=X_{j})=1/n \quad \forall i,j=1,\dots,n,\text{ and }
\Xbars\eqdef n^{-1}{{\textstyle\sum\nolimits_{i=1}^{n}}} \Xs_{i}.
\end{gather*}
Theorem \ref{theor:SFMiid} shows that the c.d.f. of $f(\bar{X}) - f (\mu)$ is uniformly well approximated by the c.d.f. of $f(\Xbars)-f(\bar{X})$ conditioned on $\{X_{i}\}_{i=1}^{n}$.
\begin{theorem}[Accuracy of the bootstrap for the Smooth Function Model]
\label{theor:SFMiid}
Let the above assumptions  and conditions of Theorem \ref{theorem:E_bootstr_l2_accur} be fulfilled. It holds with probability $\geq 1-6\ex^{-x}$ for $x>0$:
\begin{align*}
&\hspace{-1.4cm}\sup\nolimits_{t\in \R}
\left|
 \P\left(f(\bar{X}) - f (\mu) \leq t\right)
-
\Pb\bigl(f(\Xbars)-f(\bar{X})  \leq t\bigl)
\right|
\\\leq \Deltas_{f,iid}
& \eqdef 
{2C_{f,2} \Cz\sigma^{2}\tilde{C}_{x,2}}{C_{f,l}^{-1}}({\dimp}/{n})^{1/2}
+
 { \CBEBiid \Cz^{K/(K-2)} C_{M,K}} {n^{-1/2}}
\\&
+\,\CBEBiid(2\tilde{C}_{x,K})^{1/(K-2)}\left\{\Cz \sigma \right\}^{K/(K-2)}{n^{-1/2}} +  R_{1,n,K},
\end{align*}
where $C_{M,K}\eqdef  \left[\|\E(X_{1}-\mu)^{K}\|+\|\E (Y_{1}-\mu)^{K}\|\right]^{{1}/{(K-2)}}$, the term $\tilde{C}_{x,K}$ is described in the previous statement, and $R_{1,n,K}\leq Cn^{-1/2}  \Cz^{2}\sigma^{K+1}\Cx^{K-1} {C}_{x,K}$ is defined in (B.14) and (B.14) (Section B.1 in the supplement \citep{Zhilova2016supp}).  
\end{theorem}
\begin{corollaryT}
\label{corol1}
 Consider the following upper quantile functions of the bootstrap approximations: $\Quants_{2}(\alpha)\eqdef \inf \left\{t \in \R: \Pb\left(\|\Ss_{n}\|>t\right)\leq \alpha \right\},$ $\Quants_{f}(\alpha)\eqdef \inf \left\{t \in \R: \Pb\bigl(f(\Xbars)-f(\bar{X})>t\bigr)\leq \alpha \right\}$ for  \(\alpha \in (0,1)\). Let $x>0$.  Theorems \ref{theorem:E_bootstr_l2_accur} and \ref{theor:SFMiid} imply the following two bounds:
\begin{align*}
\left|\P\bigl(\|S_{n}\|> \Quants_{2}(\alpha) \bigr)-\alpha\right|
&\leq 2\Deltas_{\BCl,iid}+6\ex^{-x},\\
\left|
 \P\left(f(\bar{X}) - f (\mu) >  \Quants_{f}(\alpha)\right)
-
\alpha
\right|
&\leq 
2\Deltas_{f,iid}+6\ex^{-x}.
\end{align*}
\end{corollaryT}

\subsection{Weighted bootstrap}
\label{SECT:WB}
Let $X_{1},\dots,X_{n}$ be independent random vectors with $\Sigma_{i}\eqdef \Var(X_{i})\in \SPDp$, let also $\{X_{i}\}_{i=1}^{n}$ be sub-Gaussian, i.e. it holds for some $\sigma_{i}^{2}>0$, $\forall\,\alpha \in \R^{\dimp}$, and $\forall\, i=1,\dots,n$
$\E\bigl\{\exp(\alpha^{\T}X_{i})\bigr\}\leq
\exp \left(\|\alpha\|^{2}\sigma_{i}^{2}/2\right).$
Denote $\bar{\sigma}^{2}_{k}\eqdef n^{-1} \sum\nolimits_{i=1}^{n}\sigma_{i}^{2k}$. 
Assume that there exist i.i.d. random vectors $Y_{1},\dots,Y_{n}$ satisfying \eqref{eq:mom_2} and \eqref{cond:XiYi2} for $K=4$.  Assume also that p.d.f.-s of $X_{i}$ are bounded with a constant $c_{f}>0$. 
The bootstrap random weights \(\veps_{1},\dots,\veps_{n}\), are taken as in \eqref{def:wb_eps_intro}. These are some examples of such random weights  (here \(z_{i}\sim \mathcal{N}(0,1),\) independent of \(e_{i}, c_{i}, b_{i}\)): $(1-2^{-2/3})^{1/2}z_{i}+2^{-1/3}\left(e_{i}-1\right)$ for \(e_{i}\sim exp(1)\); $2^{-1/2}z_{i}+2^{-1}( c_{i}-1)$ for $c_{i}\sim\chi_{1}^{2}$; $(1-3^{-2/3})^{1/2}z_{i}+3^{-1/3}2(b_{i}-0.5)$ for $b_{i}\sim Bernoulli(0.5)$. 
More examples of the bootstrap weights satisfying \eqref{def:wb_eps_intro}  can be found in the works of \cite{Liu1988bootstr}  and \cite{Mammen1993bootstrap}.

The weighted bootstrap approximation of the sum $S_n$ is $\Sb_{n}\eqdef n^{-1/2}{{\textstyle\sum\nolimits_{i=1}^{n}}}X_{i} \veps_{i}.$  The probability distribution of $\Sb_n$ is taken conditioned on \(\{X_{i}\}_{i=1}^{n}\). 
\begin{theorem}[Accuracy of the weighted bootstrap for $S_{n}$ on the set $\BCl$]
\label{theorem:W_bootstr_l2_accur}
Let the above conditions be fulfilled, 
then it holds with probability $\geq 1-6\ex^{-x}$ for $x>0$:
\begin{align*}
&  \sup\nolimits_{B\in \BCl}\left|
 \P\left(S_{n}\in B \right)
-
\Pb\bigl(\Sb_{n}\in B \bigl)
\right|
\\&\leq 
\Deltab_{\BCl,w,ind}
\eqdef \CBEBniidw
\bigl\{
{ \CzL^{4}{{\textstyle\sum\nolimits_{i=1}^{n}}}
\E\left(\|X_{i}\|^{4}+\|Y_{i}\|^{4}\right)
}/{n^{2}}
\bigr\}^{1/5}
+ R_{2,3},
\end{align*}
where for  $\dimp \leq C \sqrt{n}$  it holds 
$R_{2,3}\leq  C(\CzL\vee  \bar{\sigma}_{1}^{4}\tilde{C}_{x,4} \{1+\E(\vepst_{i}^{4})\})({p}/{\sqrt{n}})^{1/3}$;  a detailed definition of $R_{2,3}$ is given in (B.17), and constant $\CBEBniidw>0$ is defined in (B.16) (Section B.2 in the supplement \citep{Zhilova2016supp}).
\end{theorem}
\begin{remark}
\label{remark:WbCz}
Lemma \ref{lemma:Yexist3} implies (see also Remark \ref{remark:Czdepend}) that for the case $K=4$, $\CzL$ in Theorem \ref{theor:NcBE_noniid} can be taken as a generic constant independent of the dimension $\dimp$.
The result in Theorem \ref{theorem:W_bootstr_l2_accur} relies on the Berry--Esseen type inequality in Theorem \ref{theor:NcBE_noniid}, and, therefore,  we can take $\CzL=const$  in the above statement. 
\end{remark}
\begin{corollaryT}
\label{coroll2}
 Consider the following upper quantile function of the approximating sum obtained using the weighted bootstrap: $\Quantb_{2,w}(\alpha)\eqdef \inf \{t \in \R: \Pb\left(\|\Sb_{n}\|>t\right)\leq \alpha \}$,  \(\alpha \in (0,1)\). Theorem \ref{theorem:W_bootstr_l2_accur} 
implies the following bound
\begin{align*}
\bigl|\P\bigl(\|S_{n}\|> \Quantb_{2,w}(\alpha) \bigr)-\alpha\bigr|
&\leq 2\Deltab_{\BCl,w,ind}+6\ex^{-x} \text{ for } x>0.
\end{align*}
\end{corollaryT}
\subsection{Some remarks about accuracy of the bootstrap procedures}
\label{sect:someremarks}
\begin{remark}[Theorems \ref{theorem:E_bootstr_l2_accur}, \ref{theorem:W_bootstr_l2_accur} in the asymptotic form]
\label{remark:thb1} If the conditions of Theorem \ref{theorem:E_bootstr_l2_accur} 
 are fulfilled and $\Cz$ is dimension-free, then taking $x=\log(2n)$, using the Borel-Cantelli lemma, and the deviation inequality for $\|X\|^{2}$ by \cite{Hsu2012tail} (see also Section B.4 in the supplement \citep{Zhilova2016supp}), we have with probability one
 \begin{align*}
& \sup\nolimits_{B\in \BCl}\left|
 \P\left(S_{n}\in B \right)
-
\Pb\bigl(\Ss_{n}\in B \bigl)
\right|
\\
&=  O\bigl(\{{\dimp^{K/(K-2)}/n}\}^{1/2}\{1+\log (n)/\dimp\}^{\frac{K}{2(K-2)}}+ (\dimp/n)^{1/2}\{1+\log(Kn)/\dimp\}^{\frac{(K-1)}{2}}\bigr)
\end{align*}
for $n\to \infty$ and $K\geq 3$.
Similarly,  given the conditions of Theorem  \ref{theorem:W_bootstr_l2_accur}, it holds with probability one
 \begin{align*}
& \sup\nolimits_{B\in \BCl}\left|
 \P\left(S_{n}\in B \right)
-
\Pb\bigl(\Sb_{n}\in B \bigl)
\right|
\\&= O\bigl(({p}/{\sqrt{n}})^{1/3}\{1+\log(n)/\dimp^{7/6}+(\dimp^{2}/n)^{1/5}\{1+\log(n)/\dimp\}^{2/5}\}
\\&\hspace{1cm}+
({p}/{\sqrt{n}})^{2/3}\{1+\log(n)/\dimp\}^{2}
 \bigr) \text{ for $\dimp \leq C n$ and $n\to \infty$.}
\end{align*}
\end{remark}

Theorem \ref{theorem:W_bootstr_l2_accur} implies that if the ratio $\dimp^{2}/{n}$ (up to  $\log n$) is small, then the weighted bootstrap approximation has a good accuracy. Lemma \ref{lemma:lowerboundbootst} below shows that the condition $\dimp^{2}/n= o(1)$ for $n\to \infty$ is necessary for the weighted bootstrap consistency.

\begin{lemma}[Necessary conditions on $\dimp$ and $n$ for the bootstrap consistency]
\label{lemma:lowerboundbootst}
Let $X_{i}\sim \mathcal{N}(0,\Id_{\dimp})$, $i=1,\dots,n$ be i.i.d. random vectors. 
Consider $S_{n}$, $S_{n}^{\ast}$, and  $S_{n}^{\sbt}$ as in Theorems \ref{theorem:E_bootstr_l2_accur} and \ref{theorem:W_bootstr_l2_accur}. Then 
\begin{itemize}
\item[(a)]
$\dimp=o(n)$ for $n\to \infty$ is necessary for $\sup_{x\in \R}|\Ps(\|S_{n}^{\ast}\|^{2} \leq x)-\P(\|S_{n}\|^{2}\leq x) | \overset{\P}{\to} 0$,   
\item[(b)] $\dimp^{2}=o(n)$ for $n\to \infty$ is necessary for 
$ \sup_{x\in \R}|\Ps(\|S_{n}^{\sbt}\|^{2}\leq x)-\P(\|S_{n}\|^{2}\leq x) | \overset{\P}{\to} 0$.
\end{itemize}
\end{lemma}

\begin{remark}
\label{remark:bk}
In Theorem \ref{theorem:W_bootstr_l2_accur}, the bootstrap weights \(\veps_{1},\dots,\veps_{n}\) satisfy the 3-rd moment condition \eqref{def:wb_eps_intro}. This is very similar to taking $K=4$ in Theorem  \ref{theorem:E_bootstr_l2_accur} for the Efron's bootstrap. It is not possible to continue the sequence of moments \eqref{def:wb_eps_intro} like $\E(\veps_{i}^{4})=1, \dots$, since the corresponding Hankel matrix 
$\begin{psmallmatrix}1& 0&1\\ 0& 1&1\\1& 1&1  \end{psmallmatrix}$ fails the criterion for solubility of the Hamburger moment problem 
(see, e.g., \cite{Akhiezer1965classical}). 
Together with Lemma \ref{lemma:lowerboundbootst} and the preceding results in Section \ref{SECT:MAINRES_BOOTSTR}, this implies that under the conditions of Theorem \ref{theorem:E_bootstr_l2_accur}, the Efron's bootstrap yields a better accuracy w.r.t. the ratio between $\dimp$ and $n$, than the considered  weighted bootstrap scheme.
\end{remark}

\begin{remark}
\label{remark:ebootstr_accuracy}
 In Theorems \ref{theorem:E_bootstr_l2_accur}-\ref{theorem:W_bootstr_l2_accur}, 
the sub-Gaussian tail behavior (condition \eqref{cond:subGauss1}) is required in order to apply concentration bounds for the higher-order moments of $X_{i}$ (see Section B.4 in the supplement \citep{Zhilova2016supp}).  In the asymptotic set-up, one can relax this condition, assuming, e.g. boundedness of  the $K$-th moments of $X_{i}, i=1,\dots,n$.
\end{remark}

%
\subsection{Weighted bootstrap for log-likelihood ratio statistics}
\label{sect:wblikelihood}
Here we consider a weighted (or a multiplier)  bootstrap procedure for estimation of quantiles of log-likelihood ratio statistics. Before describing the procedure and formulating a theoretical result, we give some necessary definitions.

Let \(\ys=(y_{1},\dots,y_{n})\) denote the data sample, \(y_{1},\dots,y_{n}\) are i.i.d. random observations from a probability space \((\Omega,\mathcal{F},\P)\). Introduce some known parametric family \(\{\P_{\thetav}\}\eqdef \{\P_{\thetav}\ll \mu_{0},\,\thetav\in \Theta \subseteq \R^{\dimp}\}\), here \(\mu_{0}\) is a \(\sigma\)-finite measure on \((\Omega,\mathcal{F})\) which dominates all \(\P_{\thetav}\) for \(\thetav\in\Theta\). 
The true data distribution \(\P\) is not assumed to belong to the family \(\{\P_{\thetav}\}\), thus our analysis includes the case when the parametric family \(\{\P_{\thetav}\}\) is misspecified. 
\(\{\P_{\thetav}\}\) induces the following (quasi)log-likelihood function for the sample \(\ys\): 
$L(\thetav)=L(\thetav,\ys)\eqdef
\log\left(\frac{d\P_{\thetav}}{d\mu_{0}}(\ys)\right).$
The target parameter \(\thetavs\) is defined by projecting the true probability distribution \(\P\) on the parametric family \(\{\P_{\thetav}\}\), using Kullback-Leibler divergence:
$\thetavs \eqdef \argmin\nolimits_{\thetav\in \Theta}\KL (\P,\P_{\thetav})= \argmax\nolimits_{\thetav\in \Theta} \E L(\thetav).$
The (quasi) maximum likelihood estimate (MLE) is defined as 
$\thetavt \eqdef \argmax\nolimits_{\thetav\in \Theta}  L(\thetav).$
Let \(\QuantL(\alpha)\) denote the upper quantile function of square root of the two times log-likelihood ratio statistic:  $\QuantL(\alpha)\eqdef \inf \bigl\{t \geq 0: \P\bigl(L(\thetavt)-L(\thetav)>t^{2}/2\bigr)\leq\alpha \bigr\}.$ 
\(\QuantL(\alpha)\) is a critical value of the likelihood-based confidence set \(\CS(\alpha)\):
\begin{gather}
\label{def:CSL}
\CS(t)\eqdef \bigl\{\thetav: L(\thetavt)-L(\thetav)\leq t^{2}/2\bigl\},\quad 
\P\left\{\thetavs \in  \CS(\QuantL(\alpha)) \right\}\geq  1-\alpha.
\end{gather}
Probability distribution of \(L(\thetavt)-L(\thetavs)\) depends on the unknown parameter \(\thetavs\) and \(\P,\) hence, in general, quantiles of \(L(\thetavt)-L(\thetavs)\) are also unknown.

Consider the weighted (or the multiplier) bootstrap procedure which allows to estimate
the distribution of \(L(\thetavt)-L(\thetavs)\).  Let \(u_{1},\dots,u_{n}\) be i.i.d. random variables:
\begin{gather*}
u_{i}\eqdef \veps_{i}+1, 
\text{ for \(\veps_{i}\) defined in \eqref{def:wb_eps_intro}, independent of $\ys$.}
\end{gather*}
The bootstrap log-likelihood function \(\Lb(\thetav)\) equals to the initial one \(L(\thetav)\) weighted with the random bootstrap weights \(u_{i}\):
\begin{gather*}
{{\textstyle
\Lb(\thetav)
\eqdef
\sum\nolimits_{i=1}^{n}
\log\left(\frac{d\P_{\thetav}}{d\mu_{0}}(y_{i})\right)u_{i}.}}
\end{gather*}
Recall that \(\Pb(\cdot)\eqdef \P(\cdot\cond \{y_{i}\}_{i=1}^{n})\) and \(\Eb(\cdot)\eqdef \E(\cdot\cond  \{y_{i}\}_{i=1}^{n})\). It holds \(\Eb \Lb (\thetav)= L(\thetav)\), therefore,
$\thetavt= \argmax\nolimits_{\thetav\in \Theta}  L(\thetav) 
=
\argmax\nolimits_{\thetav\in \Theta}  \Eb\Lb(\thetav),$ 
and the MLE \(\thetavt\) can be considered as a bootstrap analogue of the unknown target parameter \(\thetavs\).  The bootstrap likelihood ratio statistic is defined as
\begin{gather*}
\Lb(\thetavbt)-\Lb(\thetavt)\eqdef\sup\nolimits_{\thetav\in\Theta}\Lb(\thetav)-\Lb(\thetavt).
\end{gather*}
\(\Lb(\thetavbt)-\Lb(\thetavt)\) can be computed for each i.i.d. sample of the bootstrap weights \(u_{1},\dots,u_{n}\), thus we can calculate empirical probability distribution function of \(\Lb(\thetavbt)-\Lb(\thetavt)\)  and estimate its quantiles. Denote
\begin{align}
\label{def:QLbalpha}
\QuantLb(\alpha)&\eqdef \inf \bigl\{t \geq 0: \Pb\bigl(\Lb(\thetavbt)-\Lb(\thetav)>t^{2}/2\bigr) \leq  \alpha \bigr\}.
\end{align}
Theorem \ref{theor:bootst_likel} below provides a two-sided bound on the coverage error of  the likelihood  confidence set \eqref{def:CSL} based on the bootstrap quantile \(\QuantLb(\alpha)\).
Let us introduce some additional notation before stating the theorem. Denote \(\ell_{i}(\thetav)\eqdef \log\left(\frac{d\P_{\thetav}}{d\mu_{0}}(y_{i})\right)\), \(d_{0}^{2}\eqdef -\E \ell_{1}^{\prime\prime}(\thetavs)\), here $\ell^{\prime}_{i}(\thetav)\eqdef \nabla_{\thetav}\ell_{i}(\thetav).$ Take $X_{i}\coloneqq d_{0}^{-1}\ell^{\prime}_{i}(\thetavs)$. By previous definitions, such defined $\{X_{i}\}_{i=1}^{n}$ are i.i.d with zero mean. Moreover, if  conditions from Section B.3 in the supplement \citep{Zhilova2016supp} are fulfilled, then $\E(\|X_{i}\|^{4})<\infty$. 
 Let $Y_{1},\dots,Y_{n}$ be i.i.d. vectors meeting conditions  \eqref{eq:mom_2} and \eqref{cond:XiYi2} for $K=4$, and  $\Czxiv\eqdef \|\{\Var(Z_{i})\}^{-1/2}\|$.
  Now we are ready to formulate the following 
\begin{theorem}
\label{theor:bootst_likel}
If the conditions from Section B.3 in the supplement \citep{Zhilova2016supp} are fulfilled, then it holds with probability \(\geq 1-10 \ex^{-\xx}\) for $x>0$
\begin{gather}
\nonumber
\sup\nolimits_{t\geq 0}
\bigl|
\P\bigl\{
L(\thetavt)-L(\thetavs)\leq t
\bigr\}
-
\Pb\bigl\{
\Lb(\thetavbt)-\Lb(\thetavt)\leq t
\bigr\}
\bigr|
\leq \Delta_{L}, \text{ and }\\
\nonumber
\bigl|\P\bigl\{\thetavs \notin  \CS(\QuantLb(\alpha)) \bigl\}-\alpha\bigr|
\leq
2\Delta_{L}+10 \ex^{-\xx}, \text{ where}\\
\label{ineq:DeltaL1}
\Delta_{L}\leq  \CBEBniidw
\left\{
{ \Czxiv^{4} 
\E\left(\|d_{0}^{-1}\ell_{1}^{\prime}(\thetavs)\|^{4}+\|Y_{1}\|^{4}\right)
}/{n}
\right\}^{1/5}
\\\quad +R_{L,2,3}+\Czxiv C(\dimp+\xx){n^{-1/2}} ,
\nonumber
\end{gather}
where for  $\dimp \leq C \sqrt{n}$  it holds 
$R_{L, 2,3}\leq  C(\Czxiv\vee  (\nu\gmu)^{4}\tilde{C}_{x,4} \{1+\E(\vepst_{i}^{4})\})\left({p}/{\sqrt{n}}\right)^{1/3}$; 
 is defined in (B.17);  
a more detailed definition of the error term \(\Delta_{L}\) is given in (B.20), $\CBEBniidw$ is defined in (B.16) (Section B.2 in the supplement \citep{Zhilova2016supp}).
\end{theorem}
\begin{remark}
The third term in bound \eqref{ineq:DeltaL1} comes from Wilks-type approximations for the likelihood ratios $L(\thetavt)-L(\thetavs)$ and \(\Lb(\thetavbt)-\Lb(\thetavt)\) (see the proof in Section B.3 in the supplement \citep{Zhilova2016supp} for more details); the first two terms in \eqref{ineq:DeltaL1} come from the Berry--Esseen type inequality justifying the weighted bootstrap procedure on the set $\BCl$ (Theorem \ref{theorem:W_bootstr_l2_accur}). Due to the assumed sub-Gaussian tail behavior of  $d_{0}^{-1}\ell_{i}^{\prime}(\thetavs)$, the first term  is bounded from above with $C \dimp n^{-1/2}$ with large probability.  Thus, in Theorem \ref{theor:bootst_likel}  both Wilks-type bound and the higher-order Berry--Esseen type inequality yield similar ratios between $\dimp$ and $n$ in the  error of approximation $\Delta_{L}$.
\end{remark}
%
%
\section{Numerical experiments}
\label{sect:numer}
This section presents results of simulation studies, illustrating accuracy of the considered Berry--Esseen bounds and the bootstrap procedures.
\subsection{Berry--Esseen inequality}
Figure \ref{figure:BEp1LognPar} shows the c.d.f.-s of $S_{n}$, $\St_{n}$ and $\mathcal{N}(0,1)$ for the sample size $n=50$, dimension $\dimp=1$ and number $K-1=3$  of equal  moments of $S_{n}$ and $\St_{n}$.  
Similarly Figure \ref{figure:BEp} shows c.d.f.-s of $\|S_{n}\|^{2}$, $\|\St_{n}\|^{2}$ and $\chi^{2}_{\dimp}$ \  for $n=50$, $\dimp =7$ and $K-1=3$. 
Distributions of $X_{i}$ and $Y_{i}$ are described in the bottom of each of the Figures \ref{figure:BEp1LognPar} and \ref{figure:BEp}. The c.d.f.-s are obtained from $15\cdot 10^{3}$ i.i.d. samples. Both figures agree with the theoretical results about the higher order Berry--Esseen bounds: the latter approximation has a better accuracy than the Gaussian approximation.

\begin{figure}[!h]
 \centering
\caption{Distribution functions of $S_{n}$ and $\St_{n}$ for $n=50$, $\dimp=1$, $K=4$.}
\label{figure:BEp1LognPar}
\includegraphics[width=0.68\textwidth]{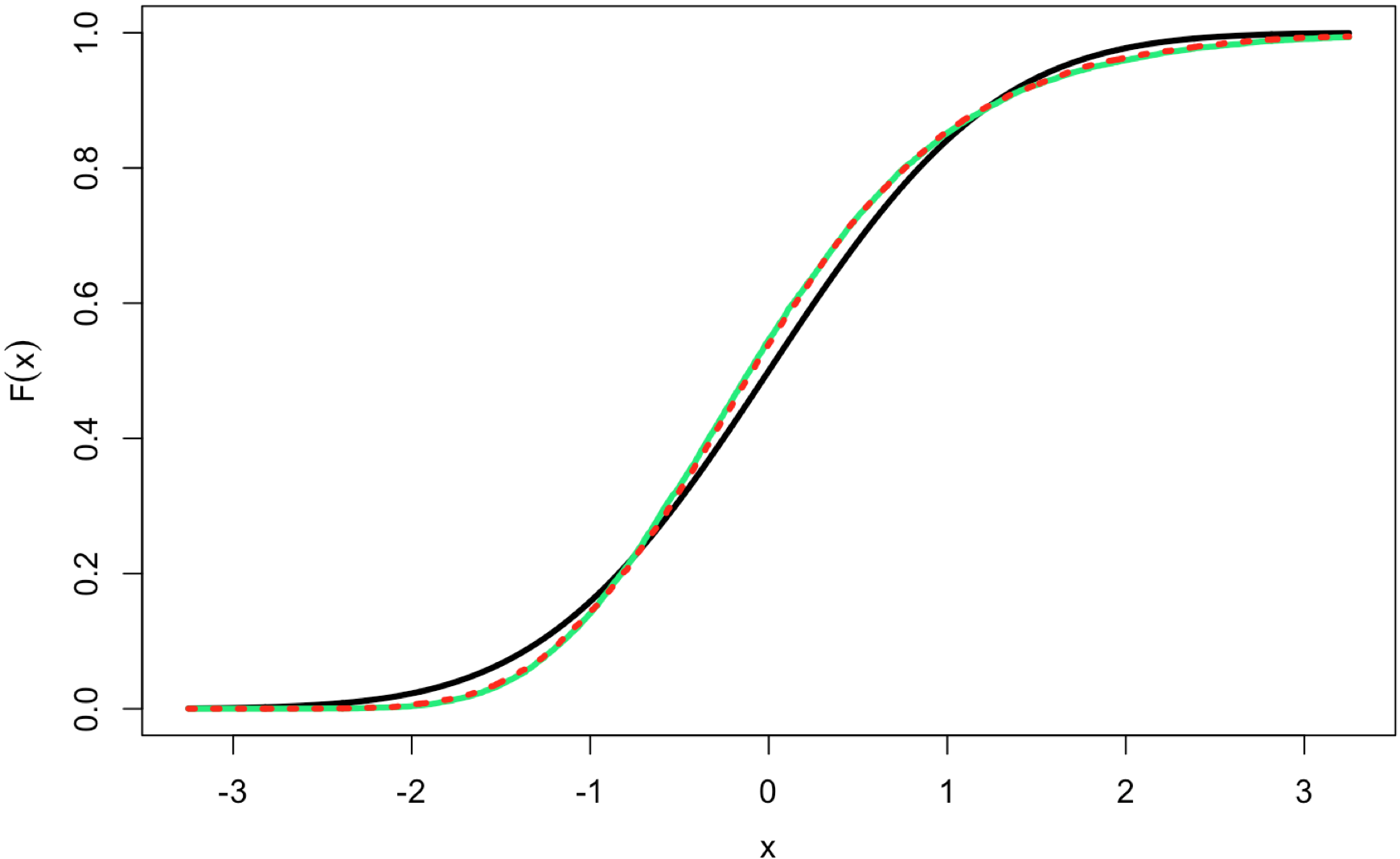}
   \par
  \begin{picture}(0,1)
\thicklines
 \linethickness{0.39mm}
\put(0,2){{\color{black}\line(1,0){15}}}
\end{picture} \hspace{0.6cm}c.d.f. of $\mathcal{N}(0,1);$\ \ 
\begin{picture}(0,1)
\thicklines
 \linethickness{0.39mm}
\put(0,2){{\color{green}\line(1,0){15}}}
\end{picture} \hspace{0.6cm}c.d.f. of $S_{n}$ for $X_{i}\sim (\ln\mathcal{N}(0,1)-1.649)/ 2.161$;
   \par
   \begin{picture}(0,1)
\thicklines
 \linethickness{0.39mm}
\put(0,2){{\color{red}\line(1,0){3}}}
\put(6,2){{\color{red}\line(1,0){3}}}
\put(12,2){{\color{red}\line(1,0){3}}}
\end{picture} \hspace{0.6cm}c.d.f. of $\St_{n}$ for $Y_{i}=\Zy_{i} +\Uy_{i}$, $\Uy_{i} \sim (Pareto(0.5,4.1)- 0.661)\cdot4.333$.
\end{figure}
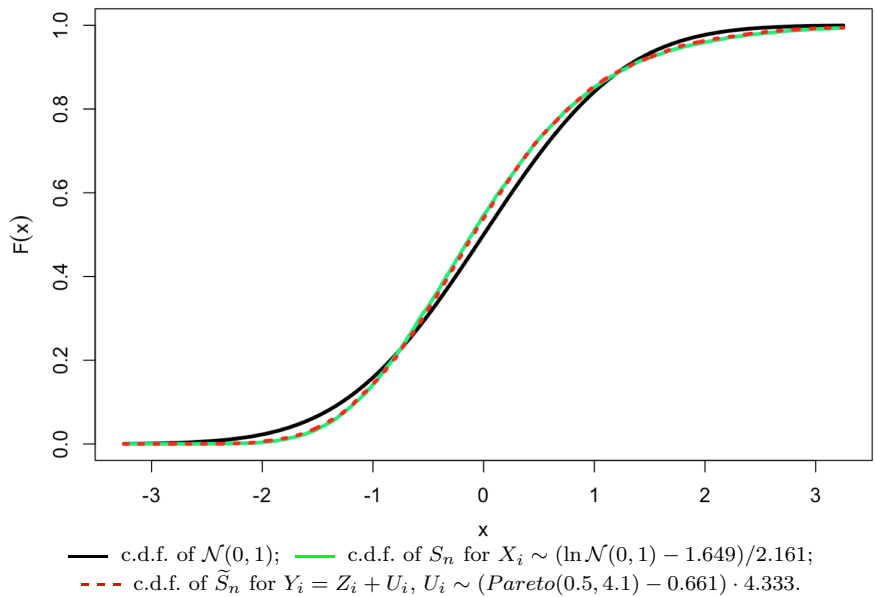

\begin{figure}[!h]
\centering
\caption{Distribution functions of $\|S_{n}\|^{2}$ and $\|\St_{n}\|^{2}$ for $n=50$, $\dimp=7$, $K=4$.}
\label{figure:BEp}
\includegraphics[width=0.68\textwidth]{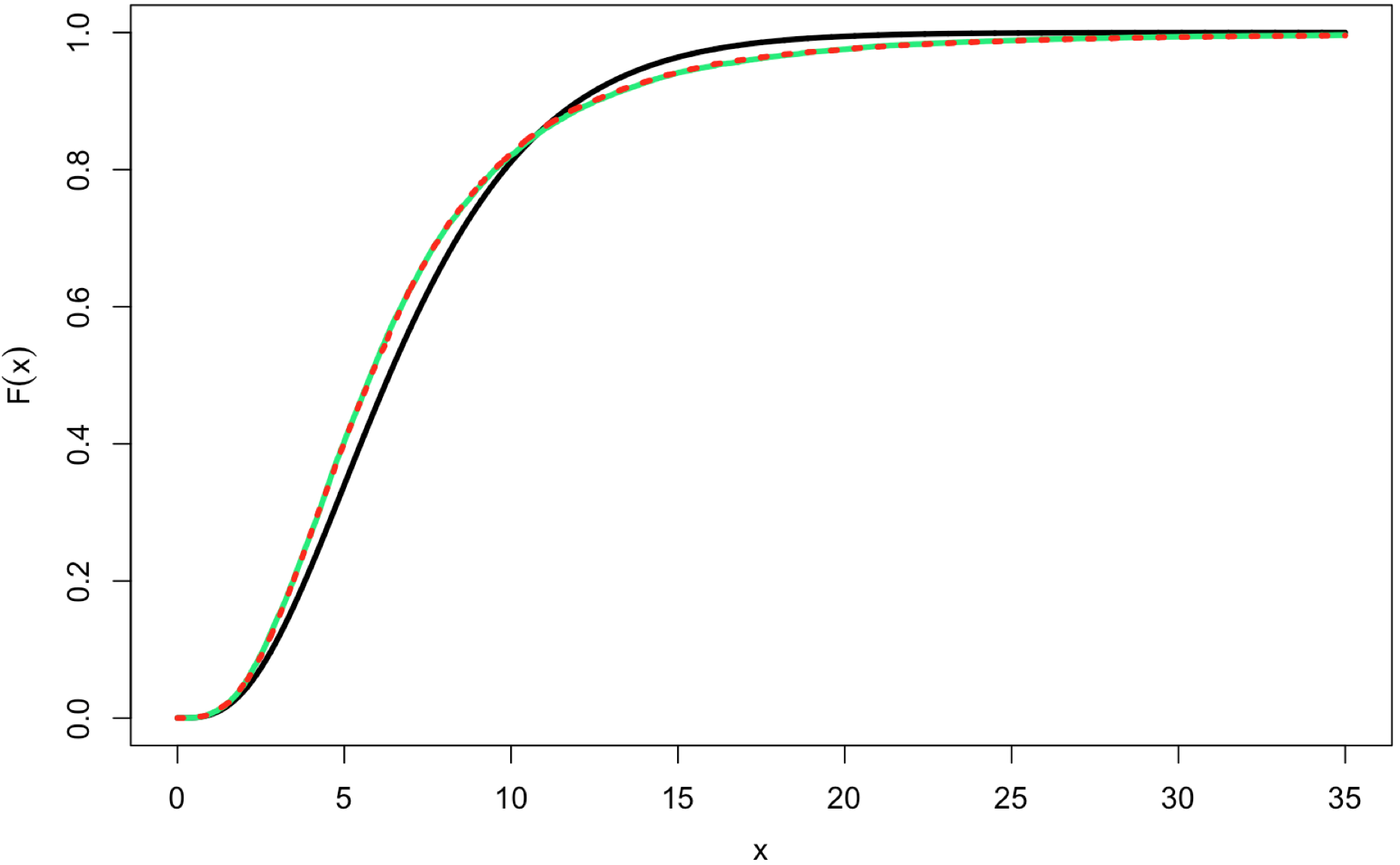}
   \begin{tabular}{ c l }
\begin{picture}(0,1)
\thicklines
 \linethickness{0.39mm}
\put(0,2){{\color{black}\line(1,0){15}}}
\end{picture}& \hspace{0.39cm}c.d.f. of $\chi^{2}_{\dimp};$\\
\begin{picture}(0,1)
\thicklines
 \linethickness{0.39mm}
\put(0,2){{\color{green}\line(1,0){15}}}
\end{picture} & \hspace{0.39cm}c.d.f. of $\|S_{n}\|^{2}$ for $X_{i}=(X_{i,1},\dots,X_{i,\dimp})^{\T},$ $X_{i,j}$ are i.i.d., \\
&
 \hspace{0.5cm} 
$X_{i,j}\sim  (\ln\mathcal{N}(0,1)-1.649)/ 2.161$;\\
\begin{picture}(0,1)
\thicklines
 \linethickness{0.39mm}
\put(0,2){{\color{red}\line(1,0){3}}}
\put(6,2){{\color{red}\line(1,0){3}}}
\put(12,2){{\color{red}\line(1,0){3}}}
\end{picture} & \hspace{0.39cm}c.d.f. of $\|\St_{n}\|^{2}$ for $Y_{i}=(\Zy_{i,1}+\Uy_{i,1},\dots,\Zy_{i,\dimp}+\Uy_{i,\dimp})^{\T},$ $\Uy_{i,j}$ are i.i.d., \\
&
 \hspace{0.5cm}  $\Uy_{i,j}\sim (Pareto(0.5,4.1)- 0.661)\cdot4.334$.
   \end{tabular}
\end{figure}

\subsection{Bootstrap}
Here we examine accuracy of the bootstrap procedures for $\|S_{n}\|$ (described in Section \ref{SECT:MAINRES_BOOTSTR}) by computing coverage probabilities using bootstrap quantiles $\Quantb_{2}(\alpha)$. All the results are collected in Table \ref{tab:bootstrcoverage}. Columns $n$, $\dimp$, \({\LL(\epsc_{i})}\),  $\LL(X_{i,j})$  show the sample size, the dimension, the distribution of the bootstrap weights $\epsc_{i}$, and the distribution of  $X_{i,j}$, where i.i.d. coordinates $X_{i,j}$ are s.t.  $X_{i}=(X_{i,1},\dots,X_{i,\dimp})^{\T}$. Nominal coverage probabilities $1-\alpha$ are given in the second row \(0.975, 0.95, \) \(0.90, 0.85, \dots, 0.50\). All the rest numbers represent frequencies of the event  $\{\|S_{n}\|\leq \Quantb(\alpha)\}$, computed for different $n$, $\dimp$, $\alpha$, \({\LL(\epsc_{i})}\), and $\LL(X_{i,j})$, from $7\cdot 10^{3}$ i.i.d. samples $\{X_{i}\}_{i=1}^{n}$ and $\{\epsc_{i}\}_{i=1}^{n}$. We consider three types of the bootstrap weights: first one  
$\epsc_{i}=z_{i}+u_{i}$, with
 $u_{i}\sim (Bernoulli(b)-b)\sigma_{u}$, $b=0.276,\sigma_{u}\approx 2.235$, and $z_{i}\sim \mathcal{N}(0,\sigma_{z}^{2})$, $\sigma_{z}^{2}\approx0.038$, for this case $\E \epsc_{i}=0$, $\E(\epsc_{i}^{2})=\E(\epsc_{i}^{3})=1$, therefore $\epsc_{i}$ meet conditions \eqref{def:wb_eps_intro}. The second type is $\epsc_{i}\sim\mathcal{N}(0,1)$, in this case $\E(\epsc_{i}^{3})\neq1$, and the approximation accuracy corresponds to the classical normal approximation with a larger error term. In this numerical experiment we check, whether the  additional condition $\E(\epsc_{i}^{3})=1$ improves numerical performance of the weighted bootstrap for $\|S_{n}\|$. The third type of the weights $Multinom.$ corresponds to the multinomial distribution $Multinomial(n;1/n,\dots,1/n)$, i.e., to the classical Efron's bootstrap scheme. Table \ref{tab:bootstrcoverage} confirms the higher-order properties of the bootstrap schemes for most of the computed coverage probabilities.
 
 \begin{table}[!h]
\caption{
Coverage probabilities $\P\bigl(\|S_{n}\|\leq Q^{\circ}_{2}(\alpha)\bigr)$}
\vspace{-0.5cm}
\label{tab:bootstrcoverage}
\begin{center}
\smallskip\noindent
\scalebox{0.87}{
\hspace{-0.86cm}
  \begin{tabular}{|c|c|c|c| c|c|c|c|c |c|c|c| }  \cline{5-12}
  \multicolumn{4}{c|}{}&\multicolumn{8}{c|}{{{Confidence levels}}}\\\hline
  \ \(n\)\
&\ \(\dimp\)\ &\(\Scale[0.9]{\LL(X_{i,j})}\) &\({\LL}(\epsc_{i})\) &$\Scale[1]{\mathbf{0.975}}$ &\(\Scale[1]{\mathbf{0.95}}\) &$\Scale[1]{\mathbf{0.90}}$& \(\Scale[1]{\mathbf{0.85}}\)& \(\Scale[1]{\mathbf{0.80}}\)& \(\Scale[1]{\mathbf{0.70}}\)&\(\Scale[1]{\mathbf{0.60}}\)&\(\Scale[1]{\mathbf{0.50}}\)\\
\hline\hline
 \multirow{9}{*}{\(\Scale[0.9]{400}\)}
&\multirow{9}{*}{\(\Scale[0.9]{40}\)}
&
\multirow{3}{*}{\(\Scale[0.9]{\chi_{1}^{2}-1}\)}
&\(\Scale[0.9]{\LL(z_{i}+u_{i})}\)
 & \(0.982\)& \(0.957\) & \(0.910\) &\(0.855\) &\(0.804\) &  \(0.701\) &$0.595$&$0.491$
  \\ \cline{4-12}
 &
 &
 & \(\Scale[0.9]{\mathcal{N}(0,1)}\)
 & \(0.984\)& \(0.960\) & \(0.914\) &\(0.862\) &\(0.810\) &  \(0.704\)& $0.597$&$0.495$
  \\ \cline{4-12}
 &
 &
 & \(\Scale[0.9]{Multinom.}\)
 & \(0.983\)& \( 0.960\) & \(0.916\) &\(0.864 \) &\( 0.812\) &  \(0.702\)&$0.593$&$0.492$
 \\ \cline{3-12}
 &
 &
 \multirow{3}{*}{\(\Scale[0.9]{Pareto^{*}}\)}
 & \(\Scale[0.9]{\LL(z_{i}+u_{i})}\)
 & \(0.984\)& \(0.964\) & \(0.917\) &\(0.865\) &\(0.813\) &  \(0.704\)&$0.593$& $0.490$
 \\ \cline{4-12}
 &
 &
 & \(\Scale[0.9]{\mathcal{N}(0,1)}\)
 & \(0.986\)& \(0.972\) & \(0.925\) &\(0.873\) &\( 0.821\) &  \(0.707\)&$0.589$&$0.480$
   \\ \cline{4-12}
 &
 &
 & \(\Scale[0.9]{Multinom.}\)
 & \(0.989\)& \(0.969 \) & \(0.927\) &\( 0.875 \) &\( 0.822\) &  \(0.710\)&$0.591$&$0.475$
 \\ \cline{3-12}
 &
 &
 \multirow{3}{*}{\(\Scale[0.9]{\ln\mathcal{N}^{\ast}(2.5)}\)}
 & \(\Scale[0.9]{\LL(z_{i}+u_{i})}\)
 & \(0.996\)& \(0.987\) & \(0.958\) &\(0.912\) &\(0.863\) & \(0.711\)&$0.555$& $0.416$
 \\ \cline{4-12}
 &
 &
 & \(\Scale[0.9]{\mathcal{N}(0,1)}\)
 & \(0.998\)& \(0.992\) & \(0.973\) &\(0.934\) &\(0.880\) &  \(0.725\)&$0.543$& $0.387$
  \\ \cline{4-12}
 &
 &
 & \(\Scale[0.9]{Multinom.}\)
 & \(0.998\)& \( 0.994 \) & \(0.967\) &\(0.914 \) &\( 0.847\) &  \(0.678\)&$0.511$&$0.390$
  \\ \hline\hline
  \multirow{9}{*}{\(\Scale[0.9]{150}\)}
&\multirow{9}{*}{\(\Scale[0.9]{15}\)}
&
\multirow{3}{*}{\(\Scale[0.9]{\chi_{1}^{2}-1}\)}
&\(\Scale[0.9]{\LL(z_{i}+u_{i})}\)
 & \(0.983\)& \(0.958\) & \(0.907\) &\(0.855\) &\(0.807\) &  \(0.703\)&$0.596$&$0.492$
  \\ \cline{4-12}
 &
 &
 & \(\Scale[0.9]{\mathcal{N}(0,1)}\)
  & \(0.986\)& \(0.965\) & \(0.915\) &\(0.863\) &\(0.811\) &  \(0.706\) &$0.595$&$0.485$
   \\ \cline{4-12}
 &
 &
 & \(\Scale[0.9]{Multinom.}\)
 & \(0.986\)& \(0.964\) & \(0.912\) &\(0.855 \) &\( 0.826\) &  \(0.683\)&$ 0.576$&$0.468$
 \\ \cline{3-12}
 &
 &
 \multirow{3}{*}{\(\Scale[0.9]{Pareto^{*}}\)}
 & \(\Scale[0.9]{\LL(z_{i}+u_{i})}\)
 & \(0.985\)& \( 0.967\) & \(0.920\) &\(0.869\) &\(0.807\) &  \(0.695\)&$0.585$&$0.472$
 \\ \cline{4-12}
 &
 &
 & \(\Scale[0.9]{\mathcal{N}(0,1)}\)
 & \(0.990\)& \(0.974\) & \(0.931\) &\(0.882\) &\(0.820\) &  \(0.697\)&$0.580$&$ 0.459$
   \\ \cline{4-12}
 &
 &
 & \(\Scale[0.9]{Multinom.}\)
 & \(0.988\)& \(0.973 \) & \(0.926\) &\(0.866\) &\(0.804\) &  \(0.683\)&$0.561$&$0.449$
 \\ \cline{3-12}
 &
 &
 \multirow{3}{*}{\(\Scale[0.9]{\ln\mathcal{N}^{\ast}(2.5)}\)}
 & \(\Scale[0.9]{\LL(z_{i}+u_{i})}\)
 & \(0.992\)& \( 0.978\) & \(0.936\) &\(0.889\) &\(0.830\) &  \(0.674\)&$0.514$&$0.386$
 \\ \cline{4-12}
 &
 &
 & \(\Scale[0.9]{\mathcal{N}(0,1)}\)
 & \(0.995\)& \(0.987\) & \(0.956\) &\(0.910\) &\(0.851\) &  \(0.693\)&$0.507$&$0.357$
   \\ \cline{4-12}
 &
 &
 & \(\Scale[0.9]{Multinom.}\)
 & \(0.996\)& \(0.987 \) & \( 0.949\) &\(0.891 \) &\(0.818\) &  \( 0.656\)&$0.494$&$0.349$
  \\ \hline\hline
  \multirow{9}{*}{\(\Scale[0.9]{50}\)}
&\multirow{9}{*}{\(\Scale[0.9]{5}\)}
&
\multirow{3}{*}{\(\Scale[0.9]{\chi_{1}^{2}-1}\)}
&\(\Scale[0.9]{\LL(z_{i}+u_{i})}\)
 & \(0.985\)& \(0.961\) & \(0.906\) &\(0.853\) &\( 0.798\) &  \(0.688\) &$0.582$&$0.483$
 
  \\ \cline{4-12}
 &
 &
 & \(\Scale[0.9]{\mathcal{N}(0,1)}\)
 & \(0.988\)& \(0.969\) & \(0.915\) &\(0.862\) &\(0.804\) &  \(0.688\)&$ 0.572$&$0.466$
   \\ \cline{4-12}
 &
 &
 & \(\Scale[0.9]{Multinom.}\)
 & \(0.985\)& \(0.959\) & \( 0.900\) &\(0.845\) &\(0.785\) &  \(0.667\)&$0.555$&$0.454$
 \\ \cline{3-12}
 &
 &
 \multirow{3}{*}{\(\Scale[0.9]{Pareto^{*}}\)}
 & \(\Scale[0.9]{\LL(z_{i}+u_{i})}\)
 & \(0.983\)& \( 0.960\) & \(  0.911\) &\( 0.852\) &\( 0.795\) &  \( 0.675\)&$0.560$&$0.460$
 \\ \cline{4-12}
 &
 &
 & \(\Scale[0.9]{\mathcal{N}(0,1)}\)
 & \(0.986\)& \(0.967\) & \( 0.923\) &\(0.866\) &\(0.804\) &  \(0.673\)&$0.546$&$0.432$
   \\ \cline{4-12}
 &
 &
 & \(\Scale[0.9]{Multinom.}\)
 & \(0.984\)& \(0.960\) & \(0.908\) &\(0.844\) &\(0.777\) &  \(0.650\)&$0.528$&$0.424$
 \\ \cline{3-12}
 &
 &
 \multirow{3}{*}{\(\Scale[0.9]{\ln\mathcal{N}^{\ast}(1.5)}\)}
 & \(\Scale[0.9]{\LL(z_{i}+u_{i})}\)
 & \(0.977\)& \(0.956\) & \( 0.903\) &\(0.839\) &\(0.775\) & $0.638$&$0.532$&  \(0.411\)
 \\ \cline{4-12}
 &
 &
 & \(\Scale[0.9]{\mathcal{N}(0,1)}\)
 & \(0.983\)& \( 0.965\) & \(0.920\) &\( 0.858\) &\(0.795\) &  \(0.645\)&$ 0.506$&$0.382$

 \\ \cline{4-12}
 &
 &
 & \(\Scale[0.9]{Multinom.}\)
 & \(0.980\)& \(0.958\) & \(0.901\) &\(0.833\) &\(0.763\) &  \(0.621\)&$0.493$&$0.383$
  \\ \hline
   \end{tabular}
   }
\end{center}
$\textit{Here $Pareto^{\ast}$ and $\ln\mathcal{N}^{\ast}(\sigma^{2})$ denote zero mean distributions $Pareto(0.5,4.1)- 0.661$}$ $\textit{and $\ln\mathcal{N}(0,\sigma^{2})-\ex^{\sigma^{2}/2}$ correspondingly.}$
\end{table}

%
%
\section*{Acknowledgments}
I am thankful to Prof. Vladimir Koltchinskii for valuable comments; I would like to thank the Editor, an Associate Editor, and anonymous Referees for careful reading of the manuscript and useful remarks which helped to improve the paper.

\begin{supplement}
\label{suppA}
\stitle{Supplement to ``Nonclassical Berry--Esseen inequalities and accuracy of the bootstrap''}
\sdescription{The supplementary material contains proofs of the results from Sections \ref{SECT:MAINRES_BE} and \ref{SECT:MAINRES_BOOTSTR}.}
\end{supplement}
\bibliographystyle{apalike}
\bibliography{references_1,ref_appr.bib}

\end{document}